\documentclass[10pt,reqno]{amsart}
\usepackage{amssymb}
\usepackage{enumitem}
\usepackage{amsfonts}
\usepackage{eucal}
\usepackage{txfonts}
\usepackage{amsmath}
\usepackage{comment}
\usepackage{hyperref}
\usepackage{extarrows}
\usepackage{mathrsfs}
\usepackage{bm}

\newtheorem{thm}{Theorem}[section]

\newtheorem{lem}[thm]{Lemma}

\theoremstyle{definition}

\theoremstyle{remark}
\newtheorem{rem}[thm]{Remark}
\numberwithin{equation}{section}

\def \R {\mathbb{R}}

\makeatletter
\@namedef{subjclassname@2020}{\textup{2020} Mathematics Subject Classification}
\makeatother

\allowdisplaybreaks[2]
\providecommand{\Ent}{\operatorname{Ent}}
\providecommand{\cL}{\mathcal L}

\providecommand{\Sn}{\mathbb S}

\begin{document}
	\title[]
	{Power Series Expansions for Logarithmic Sobolev, $\mathcal{W}$-Functionals and Rigidity 
	Theorems}
	
	\author{Liang Cheng}

	

	
	\thanks{Liang Cheng's  Research partially supported by
		Natural Science Foundation of China 12171180
	}
	
	\address{School of Mathematics and Statistics, and Key Laboratory of Nonlinear Analysis $\&$ Applications (Ministry of Education), Central  China Normal University, Wuhan, 430079, P.R.China}
	
	\email{chengliang@ccnu.edu.cn }

		\begin{abstract}
In this paper, we obtain power series expansions of the logarithmic
Sobolev and $\mathcal W$-functionals for appropriately chosen test
functions. We use these expansions to establish comparison theorems
for the optimal logarithmic Sobolev functional $L$ and Perelman's
$\boldsymbol\mu$-functional. Let $V$ be a nonempty open subset of a smooth Riemannian manifold $(M^n,g)$ and  $M_K^n$ denote the \emph{complete simply connected} space
form of sectional curvature $K$ .
Suppose
\begin{equation}
 \begin{split}
  \operatorname{Sc}_g&\ge n(n-1)K\quad\text{on }V,\\
  L(V,g,t)&\ge L(M_K^n,g_K,t)-o(t^2)\quad(t\downarrow0),
 \end{split}
 \end{equation}
\textbf{Then} $$\operatorname{Sec}_g\equiv K$$
on $V$.
If
\begin{equation}\label{mu_model_assumption}
 \begin{split}
 \boldsymbol\mu(V,g,t)&\ge\boldsymbol\mu(M_K^n,g_K,t)-o(t^2)
 \qquad(t\downarrow0),
  \end{split}
\end{equation}
then
\begin{equation}\label{mu_model_conclusion}
 |\operatorname{Rm}_g|^2\le2n(n-1)K^2\quad\text{on }V.
\end{equation}
In particular, $V$ is flat when $K=0$. 
We also get the following rigidity which can be view as analogue with rigidity part of Bishop-Gromov theorem with respect to scalar curvature: Let $V$ be a nonempty open subset of a smooth $n$-dimensional
Riemannian manifold with $n\ge 3$. Suppose that
\begin{equation}\label{scalar_curvature_lowerbound}
 \operatorname{Sc}(x)\ge n(n-1)K\quad\text{for all }x\in V,
\end{equation}
and that, for all sufficiently small $\beta>0$, the isoperimetric profile of $V$ satisfies
\begin{equation}\label{comparison_iso}
 \operatorname{I}(V,\beta)\ge
 \operatorname{I}(M_K^n,\beta)-o(\beta^{\frac{n+3}{n}}).
\end{equation}
\textbf{Then} $$\operatorname{Sec}_g\equiv. K$$
Finally, we obtain local rigidity results for nearly sharp
Euclidean logarithmic Sobolev inequalities, extending the
perspective of the classical rigidity results of
Bakry--Concordet--Ledoux and Ni.
\end{abstract}
	\maketitle	
	
	\section{Introduction}
	The celebrated Bishop-Gromov volume comparison theorem states that 
		\emph{if $(M, g)$ is a Riemannian manifold, $p \in M$, and the geodesic ball $B(p, r_0)$ is compactly contained in $M^n$ (i.e., $B(p, r_0) \Subset M$) with Ricci curvature satisfying 
		\begin{equation*}
			\operatorname{Rc} \ge (n-1) K 	\text{ on } B(p,r_0),
		\end{equation*}	
 then the function
			$$
			r \mapsto \frac{\operatorname{Vol}\bigl(B(p, r)\bigr)}{\operatorname{Vol}\bigl(B^K(p_K, r)\bigr)}
			$$
			is non-increasing in $r \in (0, r_0]$, where $B^K(p_K, r)$ is a geodesic ball of radius $r$ in the space form $M_K^n$ of constant curvature $K$. In particular,
			$
			\operatorname{Vol}\bigl(B(p, r_0)\bigr) \leq \operatorname{Vol}\bigl(B^K(p_K, r_0)\bigr)
			$
			with equality holding if and only if $B(p, r_0)$ is isometric to $B^K(p_K, r_0)$.}
			
			The Bishop-Gromov theorem implies the following rigidity result:
If $B(p,r_0)\Subset M$ satisfies
\begin{equation}\label{Ric_lower_bound}
	\operatorname{Rc} \ge (n-1) K 	\text{ on } B(p,r_0)
\end{equation}
and if its volume satisfies
\begin{equation}\label{volume}
	\operatorname{Vol}(B(p,r_0))\ge \operatorname{Vol}_K(B^K(p_K,r_0)),
\end{equation}
\textbf{then} 
$B(p,r_0)$ is isometric to $B^K(p_K,r_0)$. 
In this paper, we establish a scalar curvature analogue of this
rigidity principle.

\begin{thm}\label{rigidity_iso_profile}
Let $V$ be a nonempty open subset of a smooth $n$-dimensional
Riemannian manifold with $n\ge 3$. Suppose that
\begin{equation}\label{scalar_curvature_lowerbound}
 \operatorname{Sc}(x)\ge n(n-1)K\quad\text{for all }x\in V,
\end{equation}
and that, for all sufficiently small $\beta>0$, the isoperimetric profile of $V$ satisfies
\begin{equation}\label{comparison_iso}
 \operatorname{I}(V,\beta)\ge
 \operatorname{I}(M_K^n,\beta)-o(\beta^{\frac{n+3}{n}}).
\end{equation}
\textbf{Then} $$\operatorname{Sec}_g\equiv K$$
on $V$. Here, $M_K^n$ denote the \emph{complete simply connected} space
form of sectional curvature $K$ and
\[
\operatorname{I}(V,\beta):=
\inf_{\substack{\Omega\Subset V,\ \partial\Omega\text{ smooth}\\
		\operatorname{Vol}_g(\Omega)=\beta}}
\operatorname{Area}_g(\partial\Omega).
\]
\end{thm}

\begin{rem}\label{average_scalar_remark}
For bounded subset $V$,  condition
\eqref{scalar_curvature_lowerbound} in
Corollary~\ref{rigidity_iso_profile} can be replaced by
\begin{equation}\label{integral_scalar_curvature_lowerbound}
 \fint_V\operatorname{Sc}\,d\mu_g\ge n(n-1)K.
\end{equation}
Indeed, the first-order comparison in Theorem~\ref{key_key} gives
$\operatorname{Sc}\le n(n-1)K$ everywhere; the average condition then
forces $\operatorname{Sc}\equiv n(n-1)K$. 
\end{rem}

The rigidity properties of Riemannian manifolds with lower scalar curvature bounds are an important subject of intensive study.  
Two fundamental approaches to establishing rigidity theorems for scalar curvature are the minimal surface technique and the Dirac operator methods. 
Notable results obtained through these approaches include:
the Positive Mass Theorem by Schoen and Yau \cite{STY,STY4} and Witten \cite{W},
the non-existence of positive scalar curvature metrics on tori $\mathbb{T}^n$ by Schoen and Yau  \cite{STY2}, and Gromov and Lawson \cite{GL2},
 Llarull's rigidity theorem \cite{L2}, also see e.g.  \cite{GL}, \cite{MinOo}, \cite{ST}, etc.	For a comprehensive overview of the rigidity results obtained through these approaches, we refer to the survey \cite{Brendle} and the lectures \cite{Gromovlecture}, along with additional references therein.

Regarding  rigidity properties related  scalar curvature and isoperimetric constant,  we first note that 
when $K = 0$ and $V = M$ is a complete Riemannian manifold with bounded sectional curvature, Corollary~\ref{rigidity_iso_profile} follows as a simple consequence of the Ricci flow and Perelman's $\mathcal{W}$-functional.
Indeed, by using the monotonicity of Perelman's $\mathcal{W}$-functional of the Ricci flow directly, one can prove that if $M^n$ is a complete Riemannian manifold with bounded sectional curvature, then for any $T > 0$,
\begin{equation}\label{bingwwang_1}
	\nu(M,g,T)  \leq 0,
\end{equation}
with equality holding if and only if $M^n$ is isometric to the Euclidean space $\mathbb{R}^n$ (cf.~Proposition~4.9 in~\cite{w1}).
Here,  $\nu(M,g,T)$ is defined as follows: For any open subset $\Omega \subset M$,  recall that Perelman's $\mathcal{W}$-functional is defined as (cf. \cite{P1})
\begin{equation}\label{Perelman_W_o}
	\mathscr{W}(\Omega, g, f, t) := \int_{\Omega} \left\{ t\left(\operatorname{Sc} \cdot f^2 + 4|\nabla f|^2\right) - f^2 \log f^2 \right\} d\mu -n - \frac{n}{2} \log (4 \pi t) ,
\end{equation}
and  
$$
	\nu(\Omega,g,T) := \inf_{t \in (0,T]} \inf_{f \in \mathcal{S}(\Omega)} \mathscr{W}(\Omega, g, f, t) ,
$$
where
$$
\mathcal{S}(\Omega) := \mathcal S_W(\Omega).
$$
Additionally, note that $\mathrm{Sc} \geq 0$ on $M^n$, together with the isoperimetric constant
$$
\operatorname{I}(M) := \inf_{\Omega \subset M} \frac{\mathrm{Area}(\partial \Omega)}{\mathrm{Vol}(\Omega)^{\frac{n-1}{n}}} \geq \operatorname{I}(\mathbb{R}^n),
$$
implies $\nu(M,g,T) \geq 0$ (cf.~Lemma~3.5 in~\cite{w1}). Consequently, this implies that Corollary \ref{rigidity_iso_profile} for the case $K = 0$ and $V = M$ is a complete Riemannian manifold with bounded sectional curvature.
When $V$ is just an open subset of $M^n$,  $\boldsymbol{\nu}(V, g, t)$ does not have exact monotonicity
under the Ricci flow (cf.~Theorem~5.2 in~\cite{w1}). However, using local estimates obtained from the backward heat kernel of the Ricci flow, the author proved in a previous work~\cite{Cheng} that Corollary~\ref{rigidity_iso_profile} holds for the case $K=0$.

Unlike the aforementioned methods, in this paper we employ the power series expansions of logarithmic Sobolev and $\mathcal{W}$-functionals to prove Corollary \ref{rigidity_iso_profile}.
Precisely, we consider the following logarithmic Sobolev functional:
\begin{equation}\label{Log_sobole_func}
	\begin{aligned}
		\mathcal{L}(V,g, u,t) = &4t \int_V |\nabla u|^2\,d\mu - \int_V u^2 \log u^2\,d\mu \\
		&+ \int_V u^2\,d\mu \log\left(\int_V u^2\,d\mu\right) - \left(n + \frac{n}{2}\log(4\pi t)\right)\int_V u^2\,d\mu,
	\end{aligned}
\end{equation}
and Perelman's $\mathcal{W}$-functional:
\begin{equation}\label{Perelman_W}
	\mathcal{W}(V,g, u,t) = \mathcal{L}(V,g, u,t) + t\int_V \operatorname{Sc} \cdot u^2\,d\mu,
\end{equation}
where $\operatorname{Sc}$ is the scalar curvature of the Riemannian metric $g$ on $V$.
The logarithmic Sobolev inequality for $\mathbb{R}^n$ states that for all $f \in W_0^{1,2}(\mathbb{R}^n)$ with $\int_{\mathbb{R}^n} f^2\,d\mu = 1$,
\begin{equation*}
	\int_{\mathbb{R}^n} f^2 \log f^2\,d\mu \leq \frac{n}{2} \log\left(\frac{2}{n\pi e} \int_{\mathbb{R}^n}|\nabla f|^2\,d\mu\right),
\end{equation*}
which is equivalent to:
\begin{equation}\label{log_R_n}
	\int_{\mathbb{R}^n} f^2 \log f^2\,d\mu \leq 4t\int_{\mathbb{R}^n} |\nabla f|^2\,d\mu - n - \frac{n}{2}\log(4\pi t),
\end{equation}
for all $t > 0$ (cf.  \cite{Gross}, \cite{Weissler1978} and Lemma 8.1.7 in \cite{Topping}). Notice that by taking $f = \frac{u}{\left(\int_{\mathbb{R}^n} u^2\,d\mu\right)^{1/2}}$, \eqref{log_R_n} is equivalent to $\mathcal{L}(\mathbb{R}^n,g_{\mathbb{R}^n}, u,t) \geq 0$ for all $t>0$ and $u \in W_0^{1,2}(\mathbb{R}^n)$.
The $\mathcal{W}$-functional was introduced by Perelman \cite{P1} in his seminal work on the Ricci flow, playing a pivotal role in his celebrated proofs of Thurston's geometrization conjecture and the Poincar\'e conjecture, as well as in subsequent studies of the Ricci flow. The $\mathcal{W}$-functional defined in \eqref{Perelman_W}, under the normalization condition $\int_V u^2\,d\mu = 1$, coincides precisely with the original formulation \eqref{Perelman_W_o} used by Perelman in \cite{P1}.

To obtain the most useful power series expansions of the logarithmic Sobolev and $\mathcal{W}$-functionals, we first need select appropriate test functions. Let $(M^n,g)$ be an $n$-dimensional manifold
and $p\in V\subset \mathring{M^n}$, where $V$ is a neighborhood of $p$.
As an initial choice, one may consider the test function
$(4\pi t)^{-\frac{n}{4}}e^{-\frac{d_g(p,x)^2}{8t}}$ since it achieves the equality of logarithmic Sobolev inequality on Euclidean space. In fact, we can compute  ( see the case $a=0$ and $\alpha=0$ in Theorem \ref{expansion_L_a} ) that
\begin{equation}\label{b_1}
	\begin{aligned}
		&\quad\mathcal{L}(V,g,(4\pi t)^{-\frac{n}{4}}e^{-\frac{d_g(p,x)^2}{8t}},t)\\
		&=
		-\operatorname{Sc}(p)t-\left(\Delta \operatorname{Sc}(p)-\frac{1}{3}\operatorname{Sc}^2(p)+\frac{1}{6}|\operatorname{Rm}|^2(p)-\frac{4}{9}|\operatorname{Rc}|^2(p)\right)t^2+o(t^2),
	\end{aligned}
\end{equation}
and
\begin{equation}\label{b_2}
	\begin{aligned}
		\quad\mathcal{W}(V,g,(4\pi t)^{-\frac{n}{4}}e^{-\frac{d_g(p,x)^2}{8t}},t)=
		-
		\left(\frac{1}{6}|\operatorname{Rm}|^2(p)-\frac{4}{9}|\operatorname{Rc}|^2(p)\right)t^2+o(t^2),
	\end{aligned}
\end{equation}
where $\operatorname{Rc}$ and $\operatorname{Rm}$ denote the Ricci tensor and curvature tensor of $g$.
Due to the problematic term $-\frac{4}{9}|\operatorname{Rc}|^2(p)$ appearing in both (\ref{b_1}) and (\ref{b_2}), these expansion formulas are difficult to utilize for obtaining the desired rigidity results. 
We remark that the similar problem arises in studying the following volume conjecture for geodesic balls proposed by A.Gray
and L.Vanhecke \cite{GV}:
\textit{Suppose that all sufficiently small geodesic spheres of any point in $M^n$ have the same volume growth as in a Euclidean space, i.e.
	\begin{equation}\label{euclidean_volume}
		\mathrm{Vol}\left(B(p,r)\right)= \omega_n r^n
	\end{equation}	
	for $r\le r_p$ and all $p\in M$, where $\omega_n= \frac{\pi ^{\frac{n}{2}}}{\left(\frac{1}{2} n\right) !}$. Then $M^n$ is locally flat.}	
When $dim (M^n)\le 3$, this conjecture was proved by  A.Gray
and L.Vanhecke \cite{GV} by using of the following power series expansion for $\mathrm{Vol}\left(B(p,r)\right)$(see \cite{GV} or \cite{G}):
\begin{equation}\label{volume_expansion}
	\begin{aligned}
		&\mathrm{Vol}\left(B(p,r)\right)\\
		=& \omega_n r^n\left\{1-\frac{\operatorname{Sc}(p) }{6(n+2)}r^2\right.\\
		&\left.-\frac{1}{20(n+2)(n+4)}\left(\Delta \operatorname{Sc}(p)-\frac{5}{18}\operatorname{Sc}^2(p)+\frac{1}{6}|\operatorname{Rm}|^2(p)-\frac{4}{9}|\operatorname{Rc}|^2(p)\right)r^4+O\left(r^6\right)\right\}.
	\end{aligned}
\end{equation}
Notice that \eqref{euclidean_volume} and \eqref{volume_expansion} together imply that $\operatorname{Sc}\equiv 0$ and
$\frac{1}{6}|\operatorname{Rm}|^2\equiv\frac{4}{9}|\operatorname{Rc}|^2$. Therefore, they can conclude $\operatorname{Rm}\equiv 0$ when $dim (M^n)\le 3$.
When $dim (M^n)\ge 4$, \eqref{volume_expansion} is similarly difficult to utilize for proving the conjecture due to the presence of the problematic term $-\frac{4}{9}|\operatorname{Rc}|^2(p)$ in \eqref{volume_expansion}.

The key observation in this paper is the following: Rather than using the  test function 
$(4\pi t)^{-\frac{n}{4}}e^{-\frac{d_g(p,x)^2}{8t}}$, we consider a modified version
$$
u=(4\pi t)^{-\frac{n}{4}}e^{-\frac{d_g(p,x)^2}{8t}}\eta,
$$
where $\eta$ satisfies:
\begin{enumerate}
	\item all $\eta(\cdot,t)$ have support in a fixed compact subset
$Q\Subset V$, with $p$ in its interior, and
$\sup_{0<t<t_0}\|\eta(\cdot,t)\|_{C^1(Q)}<\infty$;
	\item on a fixed normal ball $U\Subset V$ about $p$,
\[
 \eta(x,t)^2=\phi_0(x)+t\phi_1(x)+t^2\phi_2(x)+\rho_t(x),
 \qquad\|\rho_t\|_{C^1(U)}=o(t^2),
\]
where $\phi_0\in C^4(\overline U)$, $\phi_1\in C^2(\overline U)$,
and $\phi_2\in C^1(\overline U)$;

	\item and  $\phi_0$  admits the local expansion
	$$
\phi_0(x) = 1+\frac{1}{3}\operatorname{Rc}(p)_{ij}x^ix^j + o(|x|^2) \quad \text{around } p,
$$
where $\{x^i\}$ denotes the normal coordinates on $T_pM$, and $\operatorname{Rc}(p)$ denotes the Ricci curvature with respect to $g$ at the point $p$;
\end{enumerate}
We denote by $\mathcal{B}_p(V)$ the set of all such test functions:
\begin{equation}\label{B_def}
	\mathcal{B}_p(V) = \biggl\{u(x,t) \biggm| 
	\begin{aligned}
		&u = (4\pi t)^{-\frac{n}{4}}e^{-\frac{d_g(p,x)^2}{8t}}\eta, \\
		&\text{where } \eta \text{ satisfies conditions (1), (2), and (3)}
	\end{aligned}
	\biggr\},
\end{equation}
Remarkably, by choosing $u\in \mathcal{B}_p(V)$, the problematic terms in (\ref{b_1}) and (\ref{b_2}) vanish completely! This yields the following elegant power series expansion formulas for the logarithmic Sobolev and $\mathcal{W}$-functionals. 
 
\begin{thm}\label{expansion_opt}
	Let $(M^n,g)$ be an $n$-dimensional manifold
	and $p\in V\subset \mathring{M^{n}}$, where $V$ is a neighborhood of $p$.
For   $u(x,t)\in 	\mathcal{B}_p(V)$, we have
\begin{equation}\label{expansion_nu_opt}
	\mathcal{W}(V,g,u,t)=
	-\frac{1}{6}|\operatorname{Rm}|^2(p)t^2+o(t^2),
\end{equation}		
where $\mathcal{B}_p(V)$ is defined in (\ref{B_def}).
For  $u(x,t)\in 	\mathcal{B}^{\alpha}_p(V)$, we have
	\begin{equation}\label{expansion_l_opt}
		\mathcal{L}(V,g,u,t)=
		-\operatorname{Sc}(p)t-\left(\Delta \operatorname{Sc}(p)+\frac{1}{3}\operatorname{Sc}^2(p)+\alpha \operatorname{Sc}(p)+\frac{1}{6}|\operatorname{Rm}|^2(p)\right)t^2+o(t^2),
	\end{equation}
	where
	\begin{equation*}
 	\mathcal{B}^{\alpha}_p(V)
=\mathcal{B}_p(V)\cap \{u\mid u=(4\pi t)^{-\frac{n}{4}}e^{-\frac{d_g(p,x)^2}{8t}}\eta, \frac{\partial}{\partial t}\eta^2(p,0)=\alpha \}.
	\end{equation*}	 
In  particular, if we choose $u(x,t)\in 	\mathcal{B}_p(V)\cap \{u\mid \int_V u^2 d\mu =1\}$, then
$\alpha=-\frac{1}{3}\operatorname{Sc}(p)$ and
\begin{equation}\label{expansion_n_opt}
	\mathcal{L}(V,g,u,t)=
	-\operatorname{Sc}(p)t-\left(\Delta \operatorname{Sc}(p)+\frac{1}{6}|\operatorname{Rm}|^2(p)\right)t^2+o(t^2).
\end{equation}
\end{thm}
\begin{rem}
Indeed, if we choose 
$
u = (4\pi t)^{-\frac{n}{4}} e^{-\frac{d_g(p,x)^2}{8t}} \eta,
$
where $\eta$ satisfies conditions (1) and (2), and take $\phi_0(x) = 1 + a_{ij}x^i x^j + o(|x|^2)$ around $p$, we obtain (see Theorem \ref{expansion_L_a})
\begin{equation}\label{expansion_-}
	\begin{aligned}
		\mathcal{L}(V,g, u,t) &= -\operatorname{Sc}(p)t  - \bigg(\Delta \operatorname{Sc}(p) - \frac{1}{3}\operatorname{Sc}^2(p) + 2\operatorname{tr}(a)\operatorname{Sc}(p) \\
		&\quad + \alpha \operatorname{Sc}(p) + \frac{1}{6}|\operatorname{Rm}|^2(p) - 4\left|a - \frac{1}{3}\operatorname{Rc}(p)\right|^2\bigg)t^2 + o(t^2),
	\end{aligned}
\end{equation}
where $\alpha =\frac{\partial}{\partial t}\eta^2(p,0)$, and
\begin{equation}\label{expansion_+_2}
	\begin{aligned}
		\mathcal{W}(V,g, u,t) &= -\left(\frac{1}{6}|\operatorname{Rm}|^2(p) - 4\left|a - \frac{1}{3}\operatorname{Rc}(p)\right|^2\right)t^2 + o(t^2).
	\end{aligned}
\end{equation}
The particular choice
$
a = \frac{1}{3} \operatorname{Rc}(p)
$
minimizes the coefficient of $|\operatorname{Rc}|^2(p)$ in the asymptotic expansions \eqref{expansion_-} and \eqref{expansion_+_2}.
\end{rem}
\begin{rem}
	Let $\operatorname{K}(p, T; x,t)$ be the backward heat kernel of the conjugate heat equation for the Ricci flow $g(t)$ on $M^n$, centered at $(p,T)$, i.e.,
	\begin{equation}\label{con_heat_kernel}
		\begin{cases}
			\left(-\partial_t - \Delta_{g(t)} + \operatorname{Sc}(g(t))\right) \operatorname{K}(p, T; x,t) = 0, \\
			\lim\limits_{t \nearrow T} \operatorname{K}(p, T; x,t) = \delta_{p}.
		\end{cases}
	\end{equation}
	The heat kernel $u$ has the power series expansion 
	$$
	\operatorname{K}(p, T; x,0) = (4\pi T)^{-\frac{n}{2}} e^{-\frac{d_{g(0)}(p,x)^2}{4T}} \sum\limits^N_{k=0} \psi_k T^k + O(T^{N+1-\frac{n}{2}}),
	$$
	where 
	$$
	\psi_0 = 1 + \frac{1}{3}R_{ij}(g(0))(p)x_0^ix_0^j + \frac{3}{24}\nabla_k R_{ij}(g(0))(p)x_0^ix_0^jx_0^k + O(d_{g(0)}^4(p,x))
	$$
	(see Lemma 24.16 in \cite{RFV3}). Here, $\{x_0^k\}_{k=1}^n$ are the normal geodesic coordinates  on $T_pM$ with respect to the metric $g(0)$, and $R_{ij}(g(0))(p)$ is the Ricci curvature with respect to metric $g(0)$ at $p$. 	
	Using the corresponding spatially differentiated local parametrix estimates, this implies that $\operatorname{K}(x,t)^{\frac{1}{2}}h(x):=\operatorname{K}(p, t; x,0)^{\frac{1}{2}}h(x)$ is exactly contained in $\mathcal{B}_p(V)$, where $h$ is a smooth non-negative cut-off function such that $\operatorname{supp}\{h\} \subset V$ and $h \equiv 1$ in a neighborhood of $p$. 
\end{rem}

Another purpose of this paper is to study the rigidity properties associated with the logarithmic Sobolev inequality and Perelman's $\boldsymbol{\mu}$-functional. 
Recall that Perelman's $\boldsymbol{\mu}$-functional is defined as follows (see \cite{P1}): for an open subset $V \subset M^n$,  
\begin{equation} 
 \begin{aligned}
   &L(V,g,t):=\inf\limits_{f \in \mathcal{S}(V)} \int_{\Omega} \left\{ 4t|\nabla f|^2 - f^2 \log f^2 \right\} d\mu -n - \frac{n}{2} \log (4 \pi t)  ,\\
	&\boldsymbol{\mu}(V, g, t) := \inf\limits_{f \in \mathcal{S}(V)} \mathcal{W}(V, g, f, t),  
     \end{aligned}
\end{equation}  
where $\mathcal{S}(V)$ is class of $f\in W^{1,2}_0(V)$ and $\int_V f^2d\mu=1$.  

As an application to Theorem \ref{expansion_opt}, we have the following result:

\begin{thm}\label{L_model_comparison}
Let $V$ be a nonempty open subset of a smooth Riemannian manifold $(M^n,g)$ and  $M_K^n$ denote the \emph{complete simply connected} space
form of sectional curvature $K$ .
Suppose
\begin{equation}\label{L_model_assumption}
 \begin{split}
  \operatorname{Sc}_g&\ge n(n-1)K\quad\text{on }V,\\
  L(V,g,t)&\ge L(M_K^n,g_K,t)-o(t^2)\quad(t\downarrow0),\\
  &=-n(n-1)Kt-\frac{n(n-1)}{3}K^2t^2-o(t^2)
 \end{split}
 \end{equation}
Then $\operatorname{Sec}_g\equiv K$ on $V$.
\end{thm}
\begin{rem}
For $K=0$, we have
\[
L(M_K^n,g_K,t)=0 \quad \text{for any } t>0
\]
(cf. \cite{Gross}, \cite{Weissler1978} and Lemma 8.1.7 in \cite{Topping}). For $K\neq 0$, we have (see Theorem \ref{prop:model-global})
\[
L(M_K^n,g_K,t)=-n(n-1)Kt-\frac{n(n-1)}{3}K^2t^2-o(t^2) \quad (t\downarrow 0).
\]
Indeed, since the asymptotic behavior of $L(M_K^n,g_K,t)$ as $t\to 0$ depends only on local geometry, for any nonempty subset $U\subset M_K^n$, we also have (see Remark \ref{Remark_subset})
\[
L(U,g_K,t)=-n(n-1)Kt-\frac{n(n-1)}{3}K^2t^2-o(t^2) \quad (t\downarrow 0).
\]
\end{rem}

\begin{thm}\label{mu_model_comparison}
Let $V$ be a nonempty open subset of a smooth Riemannian manifold $(M^n,g)$.
If
\begin{equation}\label{mu_model_assumption}
 \begin{split}
 \boldsymbol\mu(V,g,t)&\ge\boldsymbol\mu(M_K^n,g_K,t)-o(t^2)\\
 &=-\frac{n(n-1)}{3}K^2t^2-o(t^2)
 \qquad(t\downarrow0),
  \end{split}
\end{equation}
then
\begin{equation}\label{mu_model_conclusion}
 |\operatorname{Rm}_g|^2\le2n(n-1)K^2\quad\text{on }V.
\end{equation}
In particular, $V$ is flat when $K=0$. If $\operatorname{Sc}\ge n(n-1)K$ when $K>0$ or
$\operatorname{Sc}\le n(n-1)K$ when $K<0$ on $V$, then $\operatorname{Sec}_g\equiv K$ on $V$.
\end{thm}

As an application to Theorem \ref{mu_model_comparison}.

\begin{thm}\label{mu_rigidity_extension}
	Let $(M^n,g)$ be an $n$-dimensional manifold and $V$ an open subset of $M^n$.
	
	\noindent (i) If there exist $\tau_0>0$, $\gamma<\frac{1}{6}$ and $Q\ge 0$ such that for all $f \in W_0^{1,2}(V)$ with $\int_{V} f^2 d\mathrm{vol}=1$ and all $t\leq\tau_0$,
	\begin{equation}\label{mu_rigidity_extension_i}
		\gamma t^2\int_{V}|\operatorname{Rm}|^2f^2d\mu + \int_{V}\left\{t\left(\operatorname{Sc}\cdot f^2+4|\nabla f|^2\right)- f^2 \log f^2 \right\} d\mu - n - \frac{n}{2} \log (4 \pi t) \geq -Qt^2 - o(t^2),
	\end{equation}
	then 
	$$
	|\operatorname{Rm}|^2(x) \leq \frac{Q}{\frac{1}{6}-\gamma} \quad \text{for all } x \in V,
	$$
	and consequently $V$ must be flat if \eqref{mu_rigidity_extension_i} holds with $Q= 0$.
	
		\noindent (ii)
	In particular, by taking $\gamma=0$ in Theorem \ref{mu_rigidity_extension} (i), we have that if there exist $\tau_0>0$ and $Q\ge 0$ such that for all $t\leq\tau_0$,
	\begin{equation}\label{mu_rigidity_extension_ii}
		\boldsymbol{\mu}(V, g, t) \geq -Qt^2 - o(t^2),
	\end{equation}
	then 
	$$
	|\operatorname{Rm}|^2(x) \leq 6Q \quad \text{for all } x \in V,
	$$
	and consequently $V$ must be flat when \eqref{mu_rigidity_extension_ii} holds with $Q=0$.
\end{thm}

It was originally proved by Bakry, Concordet and Ledoux \cite{BCL}, and later by Ni \cite{Ni2} using a different method, that if an $n$-dimensional complete Riemannian manifold $(M^n,g)$ with non-negative Ricci curvature satisfies the logarithmic Sobolev inequality with the optimal constant for $\mathbb{R}^n$, then $(M^n,g)$ must be isometric to $\mathbb{R}^n$. 

As an application of  Theorem \ref{mu_rigidity_extension}, we obtain the following improved rigidity theorem concerning the logarithmic Sobolev inequality:

\begin{thm}\label{mu_rigidity}
	Let $(M^n,g)$ be an $n$-dimensional manifold. 	
 If some open subset $V\subset M$   satisfies
	\begin{equation}\label{scalar_assump}
	\int_V \operatorname{Sc} d\mu \ge 0,
	\end{equation}	
	and if the logarithmic Sobolev inequality on $V$ only differs from that of Euclidean case with $o(t^2)$, i.e.  there exist $\tau_0>0$ such that 	for all $f \in W_0^{1,2}(V)$, $\int_{V} f^2 d vol=1$ and $t\le\tau_0$,
	\begin{equation}\label{log_sobolev}
		\int_{V} f^2 \log f^2 d vol \le \int_{V}4t |\nabla f|^2 d vol-n-\frac{n}{2} \log (4 \pi t)+o(t^2),
	\end{equation}
 \textbf{then} $V$ must be flat.

\end{thm}

	The present paper is organized as follows.
Section~\ref{sec:original-expansions} gives the Gaussian moment
calculations and the proof of Theorem~\ref{expansion_opt}.
Section~\ref{sec:model-variational} is devoted to attainment of the
model variational problems and the short-time expansions of their
optimal values. Section~\ref{sec:functional-proofs} proves the
$L$- and $\boldsymbol\mu$-comparison theorems and retains the
original applications to curvature bounds and logarithmic Sobolev
rigidity. Section~\ref{sec:original-schwarz} proves the isoperimetric
corollary by quantitative Schwarz rearrangement.

\section{Power series expansion formulas of logarithmic Sobolev and \texorpdfstring{$\mathcal{W}$}{W}-functionals}\label{sec:original-expansions}	
	
In this section we calculate the power series expansion formulas of logarithmic Sobolev and $\mathcal{W}$-functionals.
 Theorem \ref{expansion_opt} follows from the following
		theorem  by choosing $a=\frac{1}{3}\operatorname{Rc}(p)$.

\begin{thm}\label{expansion_L_a}
	Let $(M^n,g)$ be an $n$-dimensional Riemannian manifold and $p\in V\subset \mathring{M^{n}}$, where $V$ is a neighborhood of $p$.
Consider test function of the form
	\begin{equation}\label{eta_condition}
u(x,t)=(4\pi t)^{-\frac{n}{4}}e^{-\frac{d_g(p,x)^2}{8t}}\eta,
	\end{equation}
where $\eta$ satisfies:
\begin{itemize}
	\item all $\eta(\cdot,t)$ have support in a fixed compact subset
$Q\Subset V$, with $p$ in its interior, and
$\sup_{0<t<t_0}\|\eta(\cdot,t)\|_{C^1(Q)}<\infty$;
	\item on a fixed normal ball $U\Subset V$ about $p$,
\[
 \eta(x,t)^2=\phi_0(x)+t\phi_1(x)+t^2\phi_2(x)+\rho_t(x),
 \qquad\|\rho_t\|_{C^1(U)}=o(t^2),
\]
where $\phi_0\in C^4(\overline U)$, $\phi_1\in C^2(\overline U)$,
and $\phi_2\in C^1(\overline U)$;
	
	\item and  $\phi_0$  admits the local expansion
	$$
	\phi_0(x) = 1+a_{ij}x^ix^j + o(|x|^2) \quad \text{around } p,
	$$
	where $\{x^i\}$ denotes the normal coordinates on $T_pM$.
\end{itemize}
Then we have
\begin{equation}\label{expansion_L}
	\begin{aligned}
		&\quad\mathcal{L}(V,g, u,t)=
	-\operatorname{Sc}(p)t\\
	&\ \ -\left[\Delta \operatorname{Sc}(p)-\frac{1}{3}\operatorname{Sc}^2(p)
       +2tr(a)\operatorname{Sc}(p)+\alpha\operatorname{Sc}(p)\right.\\
&\qquad\left.+\frac{1}{6}|\operatorname{Rm}|^2(p)
       -4\left|a-\frac{1}{3}\operatorname{Rc}(p)\right|^2\right]t^2+o(t^2),
	\end{aligned}
\end{equation}
where $\alpha=\phi_1(p)$,
and
\begin{equation}\label{expansion_nu}
			\begin{aligned}
		\quad\mathcal{W}(V,g, u,t)=-
	\left(\frac{1}{6}|\operatorname{Rm}|^2(p)-4\left|a-\frac{1}{3}\operatorname{Rc}(p)\right|^2\right)t^2+o(t^2).
\end{aligned}
\end{equation}
\end{thm}
	\begin{proof}

We take $r_0$ sufficient small such that $B(p, r_0)\subset V$ and $r_0<\operatorname{inj}(p)$.	
We first claim that 
\begin{equation}\label{asymptotic_est}
	|\mathcal{L}(B(p,r_0),g, u,t)-\mathcal{L}(V,g,u,t)|\le  C e^{-\frac{Dr^2_0}{t}},
\end{equation}
 when $t$ is sufficiently small, where $C$ and $D$ are positive constants depending on $\max |\eta|$, $\max |\nabla\eta|$ and the metric on the fixed compact support. 
We calculate that
\begin{equation}\label{2.1_1}
	\begin{aligned}
		&\left|\mathcal{L}(V,g,u,t)-\mathcal{L}(B(p,r_0),g, u,t)\right|\\
		\le &4t \int_{{V\backslash B(p,r_0)}}| \nabla u|^2d\mu+\int_{{V\backslash B(p,r_0)}} |u^2 \log u^2|d\mu +|n+\frac{n}{2}\log{4\pi t}|\int_{{V\backslash B(p,r_0)}} u^2 d\mu\\
		&+\left|\int_V u^2 d\mu\cdot \log{\int_V u^2d\mu}-\int_{B(p,r_0)} u^2 d\mu\cdot \log{\int_{B(p,r_0)} u^2d\mu}\right|\\
		\le & \int_{{V\backslash B(p,r_0)}}\frac{3d^2}{4t}u^2d\mu+8t \int_{{V\backslash B(p,r_0)}}|\nabla \eta |^2 (4\pi t)^{-\frac{n}{2}}e^{-\frac{d_g(p,x)^2}{4t}}d\mu \\
        &+\left(n+n|\log{4\pi t}|\right)\int_{{V\backslash B(p,r_0)}} u^2 d\mu\\
		&+ \int_{{V\backslash B(p,r_0)}}| \eta^2\log \eta^2 | (4\pi t)^{-\frac{n}{2}}e^{-\frac{d_g(p,x)^2}{4t}}d\mu\\
        &+\left|\int_V u^2 d\mu\cdot \log{\int_V u^2d\mu}-\int_{B(p,r_0)} u^2 d\mu\cdot \log{\int_{B(p,r_0)} u^2d\mu}\right|\\
		\le & \int_{{V\backslash B(p,r_0)}}\frac{3d^2}{4t}u^2d\mu+C' \int_{{V\backslash B(p,r_0)}} (4\pi t)^{-\frac{n}{2}}e^{-\frac{d_g(p,x)^2}{4t}}d\mu \\
        &+\left(n+n|\log{4\pi t}|\right)\int_{{V\backslash B(p,r_0)}} u^2 d\mu\\
	&+\left|\int_V u^2 d\mu\cdot \log{\int_V u^2d\mu}-\int_{B(p,r_0)} u^2 d\mu\cdot \log{\int_{B(p,r_0)} u^2d\mu}\right|,
	\end{aligned}
\end{equation}	
when $t\le 1$, where $C'$ is positive constant depending on $\max |\eta|$ and $\max |\nabla\eta|$. 
We also have
\begin{equation}\label{2.1_2}
	\begin{aligned}
		&\left|\int_V u^2 d\mu\cdot \log{\int_V u^2d\mu}-\int_{B(p,r_0)} u^2 d\mu\cdot \log{\int_{B(p,r_0)} u^2d\mu}\right|\\
		\le&\left|\int_{{V\backslash B(p,r_0)}}u^2d\mu\cdot \log{\int_{B(p,r_0)} u^2d\mu}\right|
		+\left|\int_V u^2 d\mu\cdot \log{\left(1+\frac{\int_{{V\backslash B(p,r_0)}} u^2d\mu}{\int_{ B(p,r_0)} u^2d\mu}\right)}\right|\\
		\le&  \left(\int_{{V\backslash B(p,r_0)}}u^2d\mu\right)\left( \left|\log{\int_{B(p,r_0)} u^2d\mu}\right|+\frac{\int_{V} u^2d\mu}{\int_{ B(p,r_0)} u^2d\mu}\right),
	\end{aligned}
\end{equation}	
where we used $\log{\left(1+y\right)}\le y$ for $y\ge 0$ in the last inequality. On the fixed compact support $Q$, the Riemannian volume is finite,
and the metric and volume densities in a finite coordinate cover are
uniformly controlled. Moreover,
$u^2\le Ct^{-n/2}e^{-d_g(p,x)^2/(4t)}$ and
$\frac12\le\int_{B(p,r_0)}u^2d\mu\le2$ for small $t$.
Combining with these and (\ref{2.1_1})(\ref{2.1_2}), we can conclude that
(\ref{asymptotic_est}) holds.

 Let $\Sigma_p\subset T_pM$ be the segment domain equipped with the pulled back metric $\tilde{g}=\operatorname{exp^*_p} g$ such that $\operatorname{exp_p}$ is injective on  $\Sigma_p$. We let  $\tilde{u}=u\circ \operatorname{exp_p}$ on $\Sigma_p$ and
extend 
$\tilde{u}$  on $T_pM\backslash \Sigma_p$ such that $\tilde{u}\equiv 0$ on $T_pM\backslash \Sigma_p$. Moreover, let
$d\tilde{\mu}\equiv 0$ on $T_pM\backslash \Sigma_p$.
Take $B(o,r_0)\subset T_pM$, by (\ref{asymptotic_est}) we know that
\begin{equation}\label{asymptotic_est_2}
	|\mathcal{L}(B(o,r_0),\tilde{g}, \tilde{u},t)-\mathcal{L}(V,g,u,t)|\le  c e^{-\frac{Dr^2_0}{t}},
\end{equation}
when $t$ is sufficient small. 

	Denote $\{x^k\}^n_{k=1}$ be the normal geodesic coordinates centered at $p$ on $T_pM$ with respect to metric $\tilde{g}$.	Denote $\tilde{u}^2=H^2\xi^2$ with $H^2=(4\pi t)^{-\frac{n}{2}}e^{-\frac{|x|^2}{4t}}$. Hence $\xi^2=\eta^2\circ \operatorname{exp_p}$ on $\Sigma_p$. By the assumptions, we can write 
$$
\xi^2=\phi_0+\phi_1 t+\phi_2 t^2+o(t^2),
$$
$\phi_0=1+a_{ij}x^ix^j+e_{ijk}x^ix^jx^k+b_{ijkl}x^ix^jx^kx^l+o(d^4)$, $\phi_1 =\alpha+q_ix^i+d_{ij}x^ix^j+o(d^2)$, $\phi_2=\beta+o(1)$. Here, $d:=d_g(p,x)$

Now we compute the power series expansion of $\mathcal{L}(B(o,r_0),\tilde{g}, \tilde{u},t)$. Here and below, we will use
notation $\int$ without subscript be the integral on $T_pM$ for simplicity.
We will use the following identities:
\begin{equation}
	\begin{aligned}\label{3.1_e0}
		 \int H^2 \frac{|x|^2}{t} dx^n=2n, 
	\end{aligned}
\end{equation}
for any symmetric $A_{ij}$,
\begin{equation}\label{3.1_e11}
	\begin{aligned}
		& \int H^2   A_{i j} x^i x^j dx^n \\ 
		= & t  \int_0^{\infty}(4 \pi)^{-\frac{n}{2}} e^{-\frac{r^2}{4}} r^{n+1} \int_{s^{n-1}(1)} A_{i j}  y^i y^j dy^{n-1} \\
		= & t(4\pi)^{-\frac{n}{2}} 2^{n+1}\Gamma(\frac{n}{2} +1)\int_{s^{n-1}(1)} A_{i j}  y^i y^j dy^{n-1}\\
		= & 2tr(A) t,
	\end{aligned}
\end{equation}
\begin{equation}\label{3.1_e12}
	\begin{aligned}
		& \int H^2 \frac{|x|^2}{t}  A_{i j} x^i x^j dx^n \\ 
		= & t  \int_0^{\infty}(4 \pi)^{-\frac{n}{2}} e^{-\frac{r^2}{4}} r^{n+3} \int_{s^{n-1}(1)} A_{i j}  y^i y^j dy^{n-1}\\
		= & t(4\pi)^{-\frac{n}{2}} 2^{n+3}\Gamma(\frac{n}{2} +2)\int_{s^{n-1}(1)} A_{i j}  y^i y^j dy^{n-1}\\
		= & 4(n+2)tr(A) t,
	\end{aligned}
\end{equation}
where we have used that for $A_{ij}$ is  
diagonalized
$$
\int_{s^{n-1}(1)}A_{i j}  y^i y^j =\sum_{i=1} A_{i i} \int_{s^{n-1}(1)}  (y^i)^2dy^{n-1}=\frac{\int_{s^{n-1}(1)}  \sum\limits^n_{i=1}(y^i)^2dy^{n-1}}{n}tr(A)=\frac{\pi^{\frac{n}{2}}}{\Gamma(\frac{n}{2} +1)} tr(A).
$$
 Moreover, for any four tensor $\lambda_{ijkl}$, we have
\begin{equation}\label{3.1_e21}
	\begin{aligned}
		& \int H^2   \lambda_{i j k l} x^i x^j x^k x^l dx^n\\ 
		= & t^2  \int_0^{\infty}(4 \pi)^{-\frac{n}{2}} e^{-\frac{r^2}{4}} r^{n+3} \int_{s^{n-1}(1)} \lambda_{i j k} y^i y^j y^{k } y^ldy^{n-1}\\
		= & t^2(4\pi)^{-\frac{n}{2}} 2^{n+3}\Gamma(\frac{n}{2} +2)\int_{s^{n-1}(1)} \lambda_{i j k} y^i y^j y^{k } y^ldy^{n-1}\\
		= & 4E(\lambda) t^2,
	\end{aligned}
\end{equation}
where $$
E(\lambda)\doteq \sum\limits_{i j=1}^n\left(\lambda_{i i j j}+\lambda_{i j i j}+\lambda_{i j j i}\right),$$
and
\begin{equation}\label{3.1_e22}
\begin{aligned}
	 & \int H^2 \frac{|x|^2}{t}  \lambda_{i j k l} x^i x^j x^k x^l dx^n \\ 
	 = & t^2  \int_0^{\infty}(4 \pi)^{-\frac{n}{2}} e^{-\frac{r^2}{4}} r^{n+5} \int_{s^{n-1}(1)} \lambda_{i j k} y^i y^j y^{k } y^ldy^{n-1}\\
	 = & t^2(4\pi)^{-\frac{n}{2}} 2^{n+5}\Gamma(\frac{n}{2} +3)\int_{s^{n-1}(1)} \lambda_{i j k} y^i y^j y^{k } y^ldy^{n-1}\\
	 	 = & 8(n+4)E(\lambda) t^2
	 \end{aligned}
\end{equation}
where we have used
$$
\begin{aligned}
	&\ \ \ \int_{S^{n-1}(1)} \lambda_{i j k l}  y^i y^j y^k y^l dy^{n-1}\\
	& =\frac{\pi^{\frac{n}{2}}}{(n+2)\Gamma(\frac{n}{2} +1)}\left\{3 \sum_{i=1}^n \lambda_{i i i i}+\sum_{i \neq j}\left(\lambda_{i i j j}+\lambda_{i j i j}+\lambda_{i j j i}\right)\right\}\\
	& =\frac{\pi^{\frac{n}{2}}}{(n+2)\Gamma(\frac{n}{2} +1)} E(\lambda) ,
\end{aligned}
$$
since 
$\int_{S^{n-1}(1)} y_i^4 dy^{n-1}=3 \int_{S^{n-1}(1)} y_i^2 y_j^2 dy^{n-1}=\frac{3 \pi^{n / 2}}{(n+2)\Gamma(\frac{n}{2} +1)}$ for $i\ne j$ (c.f. (A.4) and (A.5) in \cite{G} )  and 
 each the integral of which $a_i$ appears for odd times is zero because the integral over one hemisphere cancels the integral over the other.

We calculate that
\begin{equation}\label{every_term}
	\begin{aligned}
		&\quad\mathcal{L}(B(o,r_0),\tilde{g}, \tilde{u},t)\\
		=&4t \int_{B(o,r_0)}| \nabla H|^2\xi^2d\tilde{\mu}+4t \int_{B(o,r_0)}| \nabla \xi |^2H^2d\tilde{\mu}+2t \int_{B(o,r_0)}\nabla H^2 \cdot \nabla \xi^2 d\tilde{\mu}\\
        &-\int_{B(o,r_0)} H^2 \xi^2\log H^2d\tilde{\mu}\\
		& -\int_{B(o,r_0)} H^2 \xi^2\log \xi^2d\tilde{\mu}+\int_{B(o,r_0)} H^2 \xi^2 d\tilde{\mu} \log{\int_{B(o,r_0)} H^2 \xi^2 d\tilde{\mu}}\\
        &-(n+\frac{n}{2}\log{4\pi t})\int_{B(o,r_0)} H^2 \xi^2 d\tilde{\mu} \\
		=& -n\int_{B(o,r_0)} H^2 \xi^2 d\tilde{\mu}+\int_{B(o,r_0)} H^2 \xi^2 d\tilde{\mu} \log{\int_{B(o,r_0)} H^2 \xi^2 d\tilde{\mu}}+\int_{B(o,r_0)} H^2 \xi^2 \frac{|x|^2}{2t}d\tilde{\mu}\\
		&-\int_{B(o,r_0)} H^2 \xi^2\log \xi^2d\tilde{\mu}+4t \int_{B(o,r_0)}| \nabla \xi |^2H^2d\tilde{\mu}+2t \int_{B(o,r_0)}\nabla H^2 \cdot \nabla \xi^2 d\tilde{\mu},
	\end{aligned}
\end{equation}
here $d\tilde{\mu}=	\operatorname{det}\left(\tilde{g}_{k \ell}(x)\right)^\frac{1}{2} dx^n$. We next compute the  power series expansion for every term in (\ref{every_term}).

Recall that in the geodesic normal coordinates $\{x^k\}^n_{k=1}$,  $\operatorname{det}\left(\tilde{g}_{k \ell}(x)\right) $ has the following  power series expansion near $p$ (see Lemma 3.4 on p. 210 of \cite{STbook})
$$
\begin{aligned}
	\operatorname{det}\left(\tilde{g}_{k \ell}(x)\right)^\frac{1}{2} &  =1-\frac{1}{6} R_{ij}(p) x^i x^j-\frac{1}{12} \nabla_k R_{ij}(p) x^i x^j x^k +v_{ijkl} x^i x^j x^k x^l  +O\left(d^5\right),
\end{aligned}
$$		
where $v_{ijkl}=\frac{1}{24}\left(-\frac{3}{5} \nabla_k \nabla_l R_{ij}-\frac{2}{15}\sum\limits_{s, t=1}^n R_{isjt}R_{kslt}+\frac{1}{3} R_{ij} R_{k l}\right)(p)$ and $d\doteq d_g(p,x)$.  Here, $R_{ij}$
and $R_{isjt}$ denote the Ricci tensor and Riemannian curvature tensor with respect to $g$.
Then we have
\begin{equation*}
	\begin{aligned}
	\xi^2(x)\operatorname{det}\left(\tilde{g}_{k \ell}(x)\right)^\frac{1}{2} = P +G,
	\end{aligned}
\end{equation*}	
where
$$
	\begin{aligned}
	 P  :=&1+\left(a_{ij}-\frac{1}{6} R_{ij}(p)\right) x^i x^j+\left(e_{ijk}-\frac{1}{12} \nabla_k R_{ij}(p)\right) x^i x^j x^k\\
     &+\left(b_{ijkl}+v_{ijkl}-\frac{1}{6}a_{ij} R_{kl}(p)\right)x^i x^j x^k x^l\\
	& + \left(\alpha+q_ix^i+(d_{ij}-\frac{\alpha}{6}R_{ij}(p))x^ix^j\right)t+\beta t^2,
\end{aligned}
$$
and
$$
G:=o(d^2)t+o\left(d^4\right)+o\left(t^2\right).
$$
Moreover,
\begin{equation}\label{est_key}
	\begin{aligned}
		&\quad\int_{B(o,r_0)} H^2   \xi^2	\operatorname{det}\left(\tilde{g}_{k \ell}(x)\right)^\frac{1}{2}dx^n\\
		=	&
	\int_{B(o,r_0)} H^2 \left( P +G\right)dx^n\\
	=	&	\int H^2 Pdx^n-\int_{T_pM \backslash B(o,r_0)} H^2 Pdx^n+
\int_{B(o,r_0)} H^2 Gdx^n\\
	=	&	\int H^2 Pdx^n+o\left(t^2\right),
	\end{aligned}
\end{equation}
where we have used $\left|\int_{T_pM \backslash B(o,r_0)} H^2 Pdx^n\right|\le c_1 e^{-c_2\frac{r_0^2}{t}}$ when  $t$ sufficient small and $\left|\int_{B(o,r_0)} H^2 Gdx^n\right|=o\left(t^2\right)$ since $G=o(d^2)t+o\left(d^4\right)+o\left(t^2\right)$ and $G$ is bounded in $B(o,r_0)$.
By (\ref{3.1_e11}), (\ref{3.1_e21}) and (\ref{est_key}), we get
\begin{equation}\label{term1}
	\begin{aligned}
		&\quad\int_{B(o,r_0)} H^2   \xi^2	\operatorname{det}\left(\tilde{g}_{k \ell}(x)\right)^\frac{1}{2}dx^n\\
		=	&
		1+\left(2tr(a)-\frac{1}{3}\operatorname{Sc}(p)+\alpha \right)t\\
		&+\left[4\left(E(b)+E(v)-\frac{1}{6}E(a\otimes \operatorname{Rc}(p))\right)+2\left(tr(d)-\frac{\alpha}{6}\operatorname{Sc}(p)\right)+\beta \right]t^2+o(t^2),
	\end{aligned}
\end{equation}
where we used (\ref{3.1_e0})-(\ref{3.1_e22}) and
each the integral of which $x_i$ appears odd times is zero, because the integral over one hemisphere cancels the integral over the other.
It follows that
\begin{equation}\label{term2}
\begin{aligned}
	&\int_{B(o,r_0)} H^2 \xi^2 d\tilde{\mu} \log{\int_{B(o,r_0)} H^2 \xi^2 d\tilde{\mu}} \\
	=	&
\left(2tr(a)-\frac{1}{3}\operatorname{Sc}(p)+\alpha \right)t\\
	&+\left[4\left(E(b)+E(v)-\frac{1}{6}E(a\otimes \operatorname{Rc}(p))\right)+2\left(tr(d)-\frac{\alpha}{6}\operatorname{Sc}(p)\right)+\beta \right]t^2\\
	&+\frac{\left(2tr(a)-\frac{1}{3}\operatorname{Sc}(p)+\alpha \right)^2}{2}t^2+o(t^2),
\end{aligned}
\end{equation}
where we have used $f(t)\log{f(t)}=c_1t+(c_2+\frac{c_1^2}{2})t^2+o(t^2)$ if $f(t)=1+c_1t+c_2t^2+o(t^2)$.
Moreover, 
we conclude from (\ref{3.1_e0}), (\ref{3.1_e12}), (\ref{3.1_e22}) and similar arguments as (\ref{est_key}) that
\begin{equation}\label{term3}
	\begin{aligned}
		&\quad\int_{B(o,r_0)}  H^2  \frac{|x|^2}{2t} \xi^2	\operatorname{det}\left(\tilde{g}_{k \ell}(x)\right)^\frac{1}{2}dx^n\\
		=	&
		n+\left[2(n+2)tr(a)-\frac{1}{3}(n+2)\operatorname{Sc}(p)+\alpha n\right]t\\
		&+\left[4(n+4)\left(E(b)+E(v)-\frac{1}{6}E(a\otimes \operatorname{Rc}(p))\right)+2(n+2)\left(tr(d)-\frac{\alpha}{6}\operatorname{Sc}(p)\right)+\beta n\right]t^2+o(t^2),
	\end{aligned}
\end{equation}
We also have
\begin{equation*}
	\begin{aligned}
		&\quad \xi^2\log{\xi^2} \\
		=	&
		\left(\phi_0+\phi_1 t+\phi_2 t^2+o(t^2)\right)\left(\log{\phi_0}+\log{\left(1+\frac{\phi_1}{\phi_0} t+\frac{\phi_2}{\phi_0} t^2+o(t^2)\right)}\right)\\
	=	&
	\left(\phi_0+\phi_1 t+\phi_2 t^2+o(t^2)\right)\left[\log{\phi_0}+\frac{\phi_1}{\phi_0} t+\left(\frac{\phi_2}{\phi_0}-\frac{\phi_1^2}{2\phi_0^2}\right) t^2+o(t^2)\right]\\
		=	&
\phi_0\log{\phi_0}+\phi_1(1+\log{\phi_0}) t+\left(\phi_2\log{\phi_0}+\phi_2+\frac{\phi_1^2}{2\phi_0}\right) t^2+o(t^2),
	\end{aligned}
\end{equation*}
\begin{equation*}
	\begin{aligned}
		&\quad \phi_0\log{\phi_0} \\
		=	&
		\left(1+a_{ij}x^ix^j+e_{ijk}x^ix^jx^k+b_{ijkl}x^ix^jx^kx^l+o(d^4)\right)\times\\
		&\left(a_{ij}x^ix^j+e_{ijk}x^ix^jx^k+\left(b_{ijkl}-\frac{a_{ij}a_{kl}}{2}\right)x^ix^jx^kx^l+o(d^4)\right)\\
		=&a_{ij}x^ix^j+e_{ijk}x^ix^jx^k+	\left(b_{ijkl}+\frac{a_{ij}a_{kl}}{2}\right)x^ix^jx^kx^l+o(d^4),
	\end{aligned}
\end{equation*}
\begin{equation*}
	\begin{aligned}
		&\quad \phi_1(1+\log{\phi_0}) \\
		=	&
		\left(\alpha+q_ix^i+d_{ij}x^ix^j+o(d^2)\right)\left(1+a_{ij}x^ix^j+o(d^2)\right)\\
		=&\alpha+q_ix^i+\left(d_{ij}+\alpha a_{ij}\right)x^ix^j+o(d^2),
	\end{aligned}
\end{equation*}
\begin{equation*}
	\begin{aligned}
		\phi_2\log{\phi_0}+\phi_2+\frac{\phi_1^2}{2\phi_0}=\beta+\frac{\alpha^2}{2}+o(1),
	\end{aligned}
\end{equation*}
and hence
\begin{equation*}
	\begin{aligned}
		&\quad \xi^2\log{\xi^2}\operatorname{det}\left(\tilde{g}_{k \ell}(x)\right)^\frac{1}{2} \\
		=	&
	a_{ij}x^ix^j+e_{ijk}x^ix^jx^k+	\left(b_{ijkl}+\frac{a_{ij}a_{kl}}{2}-\frac{1}{6}a_{ij}R_{kl}(p)\right)x^ix^jx^kx^l+o(d^4)\\
	&+\left[\alpha+q_ix^i+\left(d_{ij}+\alpha a_{ij}-\frac{\alpha}{6}R_{ij}(p)\right)x^ix^j+o(d^2)\right]t+\left(\beta+\frac{\alpha^2}{2}+o(1)\right)t^2.
	\end{aligned}
\end{equation*}
It follows from (\ref{3.1_e11}), (\ref{3.1_e21}) and similar arguments as (\ref{est_key}) that
\begin{equation}\label{term4}
	\begin{aligned}
		&\quad \int_{B(o,r_0)}  H^2\xi^2\log{\xi^2}\operatorname{det}\left(\tilde{g}_{k \ell}(x)\right)^\frac{1}{2} dx^n \\
		=	&\left(2tr(a)+\alpha\right)t\\
		&
		+\left[2E(a\otimes a)+4E(b)-\frac{2}{3}E(a\otimes \operatorname{Rc}(p))
          +2tr(d)\right.\\
 &\qquad\left.+2\alpha tr(a)-\frac{1}{3}\alpha\operatorname{Sc}(p)
          +\beta+\frac{\alpha^2}{2}\right]t^2+o(t^2).
	\end{aligned}
\end{equation}
Since
\begin{equation*}
	\begin{aligned}
		&\quad |\nabla \xi|^2= \frac{|\nabla \xi^2|^2}{4\xi^2}\\
		&=\frac{|\nabla \phi_0|^2+tO(d)+O(t^2)}{4\xi^2}\\
		&=\frac{\sum\limits_{k=1}^n \left|\frac{\partial}{\partial x^k} (a_{ij}x^ix^j)\right|^2+o(d^2)+tO(d)+O(t^2)}{4\xi^2}\\
		&=\sum_{i=1} a_{ij}a_{ik}x^jx^k+o(d^2)+tO(d)+O(t^2),
	\end{aligned}
\end{equation*}
we get from (\ref{3.1_e11})  and similar arguments as (\ref{est_key}) that
\begin{equation}\label{term5}
	\begin{aligned}
	4t\int_{B(o,r_0)} 	|\nabla \xi|^2H^2\operatorname{det}\left(\tilde{g}_{k \ell}(x)\right)^\frac{1}{2} dx^n = 8tr(a^2)t^2+o(t^2).
	\end{aligned}
\end{equation}
Since
\begin{equation*}
	\begin{aligned}
	&\quad 2t \nabla H^2 \cdot \nabla \xi^2 \operatorname{det}\left(\tilde{g}_{k \ell}(x)\right)^\frac{1}{2}	\\
	&	= -\frac{1}{2} \left(\nabla d^2 \cdot \nabla \xi^2\right)H^2\operatorname{det}\left(\tilde{g}_{k \ell}(x)\right)^\frac{1}{2}\\
	&	= -\frac{1}{2}\tilde{g}^{rs} \frac{\partial d^2}{\partial x_r} \frac{\partial \xi^2}{\partial x_s} H^2\operatorname{det}\left(\tilde{g}_{k \ell}(x)\right)^\frac{1}{2}\\
		&	= -\frac{1}{2}\left(\delta_{rs}+\frac{1}{3}R_{rijs}(p)x^ix^j+O(d^3)\right) \frac{\partial d^2}{\partial x_r} \frac{\partial \xi^2}{\partial x_s} H^2\operatorname{det}\left(\tilde{g}_{k \ell}(x)\right)^\frac{1}{2}\\
	&=\left[-2a_{ij}x^ix^j-3e_{ijk}x^ix^jx^k-\left(4b_{ijkl}-\frac{1}{3}a_{ij}R_{kl}(p)\right)x^ix^jx^kx^l +o(d^4)\right.\\
	&\quad\left.+\left(-q_ix^i-2d_{ij}x^ix^j+o(d^2)\right)t+o(1)t^2-\frac{2}{3}\sum_r R_{ijkr}(p)a_{rl}x^ix^jx^kx^l\right] H^2,
	\end{aligned}
\end{equation*}
we get from (\ref{3.1_e11}), (\ref{3.1_e21}) and similar arguments as (\ref{est_key}) that
\begin{equation}\label{term6}
	\begin{aligned}
		&\quad \int_{B(o,r_0)}  2t \nabla H^2 \cdot \nabla \xi^2 	\operatorname{det}\left(\tilde{g}_{k \ell}(x)\right)^\frac{1}{2} dx^n \\
		&	= -4tr(a)t-\left(4tr(d)+16E(b)-\frac{4}{3}E(a\otimes \operatorname{Rc})\right)t^2+o(t^2),
	\end{aligned}
\end{equation}
where we have used the curvature symmetries and the symmetry
of $a$. Indeed, the first contraction vanishes because
$R_{iijr}=0$, while the remaining two contractions are
\[
\begin{aligned}
\sum_{i,j,r}R_{ijir}a_{rj}
&=\sum_{j,r}\operatorname{Rc}_{jr}a_{rj},\\
\sum_{i,j,r}R_{ijjr}a_{ri}
&=-\sum_{i,r}\operatorname{Rc}_{ir}a_{ri},
\end{aligned}
\]
with the curvature convention used in the normal-coordinate
expansion. They cancel, and hence
\[
E\!\left(\sum_rR_{ijkr}(p)a_{rl}\right)=0.
\]

 Combining with (\ref{term1})(\ref{term2})(\ref{term3})(\ref{term4})(\ref{term5}) and (\ref{term6}),
we have
\begin{equation}\label{star}
	\begin{aligned}
&\quad\mathcal{L}(B(o,r_0),\tilde{g}, \tilde{u},t)\\
&=-\operatorname{Sc}(p)t-\left(-20E(v)+\frac{4}{3}E(a\otimes \operatorname{Rc}(p))+2E(a\otimes a)+2\alpha tr(a)-8tr(a^2)\right.\\
&\quad\quad\quad\left.-\frac{\left(2tr(a)-\frac{1}{3}\operatorname{Sc}(p)+\alpha \right)^2}{2}+\frac{2}{3}\alpha \operatorname{Sc}(p)+\frac{1}{2}\alpha^2\right)t^2+o(t^2)\\
&=-\operatorname{Sc}(p)t-\left[-20E(v)+2tr(a)\operatorname{Sc}(p)
       +\frac{8}{3}\sum\limits_{i j=1}^n a_{ij}R_{ij}(p)\right.\\
&\qquad\left.-4tr(a^2)-\frac{1}{18}\operatorname{Sc}^2(p)
       +\alpha\operatorname{Sc}(p)\right]t^2+o(t^2)\\
&=-\operatorname{Sc}(p)t-\left(\Delta \operatorname{Sc}(p)+\frac{1}{6}|\operatorname{Rm}|^2(p)-\frac{1}{3}\operatorname{Sc}^2(p)-\frac{4}{9}|\operatorname{Rc}|^2(p)+2tr(a)\operatorname{Sc}(p)\right.\\
&\quad\quad\quad\left.+\frac{8}{3}\sum\limits_{i j=1}^n a_{ij}R_{ij}(p)-4tr(a^2)+\alpha \operatorname{Sc}(p)\right)t^2+o(t^2)\\
&=-\operatorname{Sc}(p)t-\left[\Delta\operatorname{Sc}(p)
       -\frac{1}{3}\operatorname{Sc}^2(p)+2tr(a)\operatorname{Sc}(p)
       +\alpha\operatorname{Sc}(p)\right.\\
&\qquad\left.+\frac{1}{6}|\operatorname{Rm}|^2(p)
       -4\left|a-\frac{1}{3}\operatorname{Rc}(p)\right|^2\right]t^2+o(t^2)
	\end{aligned}		
\end{equation}
where we have used
$
E(a\otimes \operatorname{Rc})=tr(a)\operatorname{Sc}(p)+2\sum\limits_{i j=1}^n a_{ij}R_{ij}(p),
$
$
E(a\otimes a)=\left(tr(a)\right)^2+2 tr (a^2)
$
and
$$
\begin{gathered}
	 E(v) =\frac{1}{24} \sum_{i j=1}^n\left\{-\frac{3}{5} \nabla_{i i} R_{j j}-\frac{6}{5} \nabla_{i j} R_{i j}+\frac{1}{3} R_{i i}R_{j j}\right. \\ \left.\quad+\frac{2}{3} R_{i j}^2-\frac{2}{15} \sum_{s, t=1}^n\left(R_{i s i t} R_{j s j t}+R_{i s j t}^2+R_{i s j t} R_{i t j s}\right)\right\} \\ =\frac{1}{360}\left(5 \operatorname{Sc}^2+8|\operatorname{Rc}|^2-3|\operatorname{Rm}|^2-18 \Delta \operatorname{Sc}\right)(p) 
	 \end{gathered}
$$
(c.f. P197 in \cite{G}). Combining (\ref{star}) with (\ref{asymptotic_est_2}), we conclude (\ref{expansion_L}) holds.
Also notice that
	\begin{equation}\label{term7}
	\begin{aligned}
	 &\quad t\int_{B(o,r_0)}  \operatorname{Sc}(\tilde{g})(x)H^2\xi^2 \operatorname{det}\left(\tilde{g}_{k \ell}(x)\right)^\frac{1}{2}dx^n	\\
		&= t\int_{B(o,r_0)} \left(\operatorname{Sc}(p)+\nabla_i\operatorname{Sc}(p)x^i+\frac{1}{2}\nabla_i\nabla_j \operatorname{Sc}(p) x^i x^j+o(d^2)\right)\\
		&\times
		\left(1+a_{ij}x^ix^j+o(d^2)+\alpha t+o(t)\right)\times (1-\frac{1}{6} R_{ij}(p) x^i x^j+o(d^2))H^2dx^n	\\
		&=t\int\left[\operatorname{Sc}(p)+\frac12\nabla_i\nabla_j\operatorname{Sc}(p)x^ix^j
        +\operatorname{Sc}(p)a_{ij}x^ix^j\right.\\
&\qquad\left.-\frac16\operatorname{Sc}(p)R_{ij}(p)x^ix^j
        +\alpha\operatorname{Sc}(p)t\right]H^2dx^n+o(t^2)\\
		&=\operatorname{Sc}(p)t+\left(\Delta \operatorname{Sc}(p)-\frac{1}{3}\operatorname{Sc}^2(p)+2tr(a)\operatorname{Sc}(p)+\alpha \operatorname{Sc}(p)\right)t^2+o(t^2)
	\end{aligned}
\end{equation}
Then (\ref{expansion_nu}) follows from (\ref{expansion_L}), (\ref{term7}).

	\end{proof}

	\begin{rem}
		From observations of (\ref{term1})--(\ref{term6}), we can see that $\mathcal{L}(V,g,u,t)$ must take the form
		\begin{equation}\label{see_why}
			\mathcal{L}(V,g,u,t) = \big(C_1\operatorname{tr}(a) + C_2R(p) + C_3\alpha\big)t + O(t^2)
		\end{equation}
		for some constants $C_1$, $C_2$, and $C_3$. 
		Direct calculations in the proof of Theorem \ref{expansion_L_a} show that $C_1 = C_3 = 0$. An alternative way to see why $C_1 = C_3 = 0$ is the following: From the logarithmic Sobolev inequality on Euclidean space, we know that $\mathcal{L}(V,g,u,t) \geq 0$ for all $u \in W_0^{1,2}(V)$ when $V$ is a Euclidean coordinate domain. This implies
		$$
		\mathcal{L}(V,g,u,t) = \big(C_1\operatorname{tr}(a) + C_3\alpha\big)t + O(t^2) \geq 0,
		$$
		for arbitrary $a$ and $\alpha$ on that Euclidean domain. Consequently, $C_1$ and $C_3$ must vanish. 
		The same argument shows that the coefficients of $\operatorname{tr}(d)$, $E(b)$, $\alpha\operatorname{tr}(a)$, and $\beta$ - which appear in calculations of the $O(t^2)$ terms of $\mathcal{L}(V,g,u,t)$ - must also be zero.
	\end{rem}
	
		\begin{rem}[An explicit normalized family]\label{explicit_optimized_family}
The corrected class is nonempty. In a sufficiently small normal ball
at $p$, take a smooth cutoff $\chi=1$ near $p$ and normalize
\begin{equation}\label{optimized_explicit_test}
 u_{p,t}=c_{p,t}(4\pi t)^{-n/4}
 \exp\!\left(-\frac{|x|^2}{8t}
             +\frac16\operatorname{Rc}_{ij}(p)x^ix^j\right)\chi(x),
 \qquad\int_Vu_{p,t}^2d\mu_g=1.
\end{equation}
The mass expansion \eqref{term1}, together with its higher smooth
Gaussian moment terms, gives
$c_{p,t}^2=1-\frac13\operatorname{Sc}(p)t+O(t^2)$ and a
$C^1$-controlled amplitude expansion through order $t^2$.
Thus Theorem~\ref{expansion_opt} applies. The same holds for a cutoff
$\chi(\langle(I-\frac43t\operatorname{Rc}(p))x,x\rangle^{1/2})$;
this latter choice makes all positive superlevel sets ellipsoids.
\end{rem}

	Now we give the proof of Theorem \ref{expansion_opt}.
	\begin{proof}[Proof of Theorem \ref{expansion_opt}]
(\ref{expansion_nu_opt}) and (\ref{expansion_l_opt}) follow from the
Theorem \ref{expansion_L_a} by choosing $a=\frac{1}{3}\operatorname{Rc}(p)$. And we see from the $O(t)$ term in (\ref{term1}) that $\alpha=-\frac{1}{3}\operatorname{Sc}(p)$ if $\int_V u^2 d\mu\equiv 1$ and  $a=\frac{1}{3}\operatorname{Rc}(p)$. So
(\ref{expansion_n_opt}) holds.
	\end{proof}

\section{Small-time expansions on the model spaces}
\label{sec:model-variational}
In this section, we compute the exact expansions of $L(M_K^n,g_K,t)$ and $\boldsymbol\mu(M_K^n,g_K,t)$ up to the $o(t^2)$ term. These expansions will be used in the subsequent proofs. The main result of this section is the following.

\begin{thm}\label{prop:model-global}
Let $M_K^n$ be the complete simply connected space form of
sectional curvature $K$. As $t\downarrow0$, we have
\begin{align}
L(M_K^n,g_K,t)
&=-n(n-1)Kt-\frac{n(n-1)}{3}K^2t^2+o(t^2),
\label{eq:LK-model}\\
\boldsymbol\mu(M_K^n,g_K,t)
&=-\frac{n(n-1)}{3}K^2t^2+o(t^2).
\label{eq:muK-model}
\end{align}
\end{thm}

\begin{rem}
For every fixed $t>0$, the infimum $L(\Omega,g,t)$ is attained
whenever $(\Omega,g)$ is a closed Riemannian manifold or
$\Omega$ is a nonempty relatively compact smooth domain,
with the admissible functions belonging to $H_0^1(\Omega)$
in the latter case. This follows from the direct method and
compact Sobolev embedding; see \cite[Chapter~7]{GT} for the
relevant compactness results, and compare
\cite[Lemma~8.1.8]{Topping} for the compact variational argument.
Constrained variation then shows that a nonnegative normalized
minimizer satisfies
\begin{equation}\label{model_Euler_original}
\begin{split}
-4t\Delta_g u-u\log u^2
&=\left(L(\Omega,g,t)+n+\frac{n}{2}\log(4\pi t)\right)u,\\
\int_\Omega u^2\,d\mu_g&=1.
\end{split}
\end{equation}
The equation holds weakly, with zero boundary trace in the
domain case.
On complete noncompact manifolds, attainment may fail, as
shown by Zhang \cite{Zhang2012}. Sufficient conditions for
existence were established by Zhang \cite{Zhang2012} and
Rimoldi--Veronelli \cite{RimoldiVeronelli2019}, under suitable
geometric assumptions and a strict gap between the optimal
value and its value at infinity.
For $\mathbb R^n$, the sharp Euclidean logarithmic Sobolev
inequality gives $L(\mathbb R^n,g_E,t)=0$, attained by the
normalized Gaussian
\[
u_t(x)=(4\pi t)^{-\frac{n}{4}}
\exp\left(-\frac{|x|^2}{8t}\right);
\]
see \cite{Gross,Weissler1978} and
\cite[Lemma~8.1.7]{Topping}.
For $\mathbb H_K^n$, $K<0$, the above existence criteria do
not apply directly: every compactly supported test function
can be moved by an isometry outside any prescribed compact
set without changing its functional value, so the optimal
value equals its value at infinity. Nevertheless,
$L(\mathbb H_K^n,g_K,t)$ is attained for every fixed $t>0$;
a direct proof is given in the appendix.
Since
$\boldsymbol\mu(M_K^n,g_K,t)=L(M_K^n,g_K,t)+n(n-1)Kt$,
the same minimizers also attain $\boldsymbol\mu$ on these
model spaces.
\end{rem}

\subsection{An estimate lemma}
For a probability measure $\nu$, write
\[
 \Ent_\nu(f^2)=\int f^2\log f^2\,d\nu
       -\left(\int f^2\,d\nu\right)\log\left(\int f^2\,d\nu\right).
\]

\begin{lem}
	\label{refm:lem:certificate}
	Let $(M^n,g)$ be a Riemannian manifold and let
	$\Omega\Subset M$ be a nonempty relatively compact open
	set. Fix $t>0$, and let
	\[
	\phi_t\in C^\infty(\overline\Omega),
	\qquad \phi_t>0,
	\qquad
	\int_\Omega\phi_t^2\,d\mu_g=1.
	\]
	Set
	\[
	d\nu_t=\phi_t^2\,d\mu_g.
	\]
	Suppose that for  $\nu_t$ and the $f\in H_0^1(\Omega)$ with $
	\int_\Omega f^2\,d\nu_t=1
	$  satisfy
	\begin{equation}\label{refm:eq:general-ref-LSI}
		\operatorname{Ent}_{\nu_t}(f^2)
		\le
		4t\int_\Omega|\nabla f|^2\,d\nu_t.
	\end{equation}
Define
	\begin{equation}\label{refm:eq:general-residual}
		\mathcal R_t
		=
		-4t\frac{\Delta_g\phi_t}{\phi_t}
		-\log\phi_t^2
		-n-\frac n2\log(4\pi t).
	\end{equation}
	Then
	\begin{equation}\label{refm:eq:general-certificate}
		\mathcal L(\Omega,g,u,t)
		\ge
		\inf_{x\in\Omega}\mathcal R_t(x).
	\end{equation}

\end{lem}

\begin{proof}
	Let $u\in C_c^\infty(\Omega)$ satisfy
	$
	\int_\Omega u^2\,d\mu_g=1,
$
	and put
	\[
	f=\frac{u}{\phi_t}.
	\]
	Since $\phi_t$ is smooth and strictly positive on
	$\overline\Omega$, we have $f\in H_0^1(\Omega)$ and
	\[
	\int_\Omega f^2\,d\nu_t
	=
	\int_\Omega u^2\,d\mu_g=1.
	\]
	Since $u=\phi_tf$, we have
	\[
	|\nabla u|^2
	=
	\phi_t^2|\nabla f|^2+
		\langle\nabla\phi_t,\nabla(\phi_tf^2)\rangle.
	\]
	Integration by parts, using $f\in H_0^1(\Omega)$,
	therefore gives
	\[
	\int_\Omega|\nabla u|^2\,d\mu_g
	=
	\int_\Omega|\nabla f|^2\,d\nu_t
	-\int_\Omega\frac{\Delta_g\phi_t}{\phi_t}
	u^2\,d\mu_g.
	\]
	Moreover, since $u^2=f^2\phi_t^2$ and
	$\int_{\Omega} f^2\,d\nu_t=1$, we have
	\[
	\begin{aligned}
		\int_\Omega u^2\log u^2\,d\mu_g
		={}&\operatorname{Ent}_{\nu_t}(f^2)+\int_\Omega f^2\log\phi_t^2\,d\nu_t.
	\end{aligned}
	\]
	Combining these identities yields
	\begin{equation}\label{refm:eq:general-identity}
		\begin{aligned}
			\mathcal L(\Omega,g,u,t)
			={}&4t\int_\Omega|\nabla f|^2\,d\nu_t
			-\operatorname{Ent}_{\nu_t}(f^2)+\int_\Omega\mathcal R_tu^2\,d\mu_g.
		\end{aligned}
	\end{equation}
	By \eqref{refm:eq:general-ref-LSI},
	\[
	\mathcal L(\Omega,g,u,t)
	\ge
	\int_\Omega\mathcal R_tu^2\,d\mu_g
	\ge
	\inf_\Omega\mathcal R_t.
	\]
\end{proof}

\subsection{Expansion on a small normal ball}

In the below, we
write $r=d_{g_K}(o,x)$ and
\[
 s_K(r)=\begin{cases}
  \sin(\sqrt K r)/\sqrt K,&K>0,\\
  r,&K=0,\\
  \sinh(\sqrt{-K}r)/\sqrt{-K},&K<0.
 \end{cases}
\]
Then
\begin{equation}\label{eq:model-radial-density}
 \begin{split}
 g_K&=dr^2+s_K(r)^2g_{\Sn^{n-1}},\\
 J_K(r)&=\left(\frac{s_K(r)}r\right)^{n-1},\qquad
 h(r)=\log J_K(r).
 \end{split}
\end{equation}

\begin{thm}\label{thm:fixed-ball}

For each $n\ge2$ and $K\in\R$, there exists $R_*>0$ such that, for
every fixed $0<R\le R_*$,
\begin{equation}\label{eq:fixed-ball-result}
 L(B_K(o,R),g_K,t)
 =-n(n-1)Kt-\frac{n(n-1)}3K^2t^2+o(t^2)
 \qquad(t\downarrow0).
\end{equation}
\end{thm}

\begin{proof}
Write $B=B_K(o,R)$ and $\lambda=(n-1)K$. By
Theorem~\ref{expansion_opt},
\[
L(B,g_K,t)
\le -n(n-1)Kt-\frac{n(n-1)}{3}K^2t^2+o(t^2),
\]
so it remains to prove the lower bound. The proof relies on Lemma~\ref{refm:lem:certificate}. The main point is to
construct a function $\varphi_t$ for which the corresponding  $\mathcal R_t(r)$ admits a uniform lower bound and has the desired
second-order expansion.  Accordingly, set
\begin{equation}\label{refm:eq:F}
F_t(r)=ar^2+br^4+ctr^2+Ar^6,
\end{equation}
where $a,b,c,A$ are independent of $t$ and will be determined
below, and define
\begin{align}
Z_t
&=(4\pi t)^{-\frac{n}{2}}
  \int_B e^{-\frac{r^2}{4t}+2F_t(r)}\,d\mu_K,
\label{refm:eq:Z}\\
\phi_t(r)
&=Z_t^{-\frac{1}{2}}(4\pi t)^{-\frac{n}{4}}
  e^{-\frac{r^2}{8t}+F_t(r)},
\qquad d\nu_t=\phi_t^2\,d\mu_K.
\label{refm:eq:phi}
\end{align}
Thus $\nu_t$ is a probability measure.

Put $d_K(r)=(n-1)s_K'(r)/s_K(r)$. Substitution into  \eqref{refm:eq:general-residual} gives
\begin{equation}\label{refm:eq:R-exact}
\begin{aligned}
\mathcal R_t(r)
={}&\log Z_t+rd_K(r)-(n-1)+2rF_t'(r)-2F_t(r)\\
&-4t\bigl(F_t''(r)+d_K(r)F_t'(r)+F_t'(r)^2\bigr).
\end{aligned}
\end{equation}
On a fixed sufficiently small normal interval $[0,R_1]$,
\begin{equation}\label{refm:eq:geo-rem}
rd_K(r)-(n-1)
=-\frac{\lambda}{3}r^2-\frac{\lambda K}{45}r^4+e_0(r),
\qquad |e_0(r)|\le C_0r^6.
\end{equation}
Moreover,
\begin{equation}\label{refm:eq:F-derivative-general}
\begin{aligned}
F_t''+d_KF_t'+F_t'^2
&=2na+2nct+D_2r^2+e_A(r,t),\\
D_2
&=4(n+2)b-\frac{2a\lambda}{3}+4a^2,
\end{aligned}
\end{equation}
where
\[
|e_A(r,t)|\le C_A(r^4+tr^2),
\qquad 0\le r\le R_1,\quad 0<t\le1.
\]
Here $C_A$ may depend on the fixed coefficients, but not on
$r$ or $t$. Since
\[
2rF_t'-2F_t=2ar^2+6br^4+2ctr^2+10Ar^6,
\]
the coefficients of $r^2$, $r^4$, and $tr^2$ in
\eqref{refm:eq:R-exact} vanish if
\[
2a=\frac{\lambda}{3},
\qquad
6b=\frac{\lambda K}{45},
\qquad
2c=4D_2.
\]
The first equality also gives
$4a^2-\frac{2a\lambda}{3}=0$. These conditions therefore
determine
\begin{equation}\label{refm:eq:abc}
a=\frac{\lambda}{6},
\qquad
b=\frac{\lambda K}{270},
\qquad
c=\frac{4(n+2)\lambda K}{135}.
\end{equation}
With these choices,
\begin{equation}\label{refm:eq:R-bound0}
\mathcal R_t(r)
=\log Z_t-\frac{4n\lambda}{3}t-8nct^2
 +10Ar^6+e_0(r)-4te_A(r,t).
\end{equation}
We now fix $A>0$ so that $10A-C_0\ge1$. The sixth-order
term does not affect the coefficient of $t^2$, but it gives
the pointwise estimate
\begin{equation}\label{refm:eq:R-bound1}
\mathcal R_t(r)
\ge \log Z_t-\frac{4n\lambda}{3}t-8nct^2
+r^6-C_A'(tr^4+t^2r^2).
\end{equation}
Indeed, writing $y=\frac{r^2}{t}$, we have
\[
r^6-C_A'(tr^4+t^2r^2)
=t^3\bigl(y^3-C_A'y^2-C_A'y\bigr)\ge-C_A''t^3.
\]
Consequently,
\begin{equation}\label{refm:eq:R-bound2}
\mathcal R_t(r)
\ge\log Z_t-\frac{4n\lambda}{3}t-8nct^2-Ct^3.
\end{equation}
 Set
$J_K(r)=\left(\frac{s_K(r)}{r}\right)^{n-1}$. The expansion
\begin{equation}\label{refm:eq:density}
\log J_K(r)
=-\frac{\lambda}{6}r^2
-\frac{\lambda K}{180}r^4+O(r^6)
\end{equation}
and \eqref{refm:eq:abc} give
\begin{equation}\label{refm:eq:combined-density}
2F_t(r)+\log J_K(r)
=\frac{\lambda}{6}r^2+\frac{\lambda K}{540}r^4
+2ctr^2+O_A(r^6).
\end{equation}
For $G_t(x)=(4\pi t)^{-\frac{n}{2}}
e^{-\frac{|x|^2}{4t}}$, the Gaussian moments
\[
\int_{\mathbb R^n}|x|^2G_t\,dx=2nt,
\qquad
\int_{\mathbb R^n}|x|^4G_t\,dx=4n(n+2)t^2
\]
yield
\[
\begin{aligned}
Z_t
={}&1+\frac{n\lambda}{3}t+\left(
\frac{n(n+2)\lambda^2}{18}
+\frac{n(n+2)\lambda K}{135}+4nc
\right)t^2+O_A(t^3).
\end{aligned}
\]
Thus, by  \eqref{refm:eq:abc},
\[
\begin{aligned}
\log Z_t-\frac{4n\lambda}{3}t-8nct^2
&=-n\lambda t
+\frac{n}{9}\bigl(\lambda^2-(n+2)\lambda K\bigr)t^2
+O(t^3)\\
&=-n(n-1)Kt-\frac{n(n-1)}{3}K^2t^2+O(t^3).
\end{aligned}
\]
Combining this with \eqref{refm:eq:R-bound2}, we obtain
\begin{equation}\label{refm:eq:R-lower}
\mathcal R_t(r)
\ge-n(n-1)Kt-\frac{n(n-1)}{3}K^2t^2-Ct^3,
\qquad 0\le r\le R.
\end{equation}

\textbf{We next verify the radial logarithmic Sobolev inequality
needed to apply the reference-measure identity.} The radial
pushforward $\bar\nu_t$ of $\nu_t$ has density proportional to
$e^{-\Psi_t(r)}$, where
\[
\Psi_t(r)
=\frac{r^2}{4t}-2F_t(r)-(n-1)\log s_K(r).
\]
Since $\left(\frac{s_K'}{s_K}\right)'
=-\frac{1}{s_K^2}$,
\begin{equation}\label{refm:eq:convexity}
\Psi_t''(r)
=\frac{1}{2t}+\frac{n-1}{s_K(r)^2}
-4a-24br^2-4ct-60Ar^4.
\end{equation}
Choose $C_s>0$ such that $s_K(r)\le C_sr$ on $[0,R_1]$.
After fixing $A$, choose $t_0\le1$ and $R_*\le R_1$ so that
\[
\frac{n-1}{C_s^2R_*^2}
\ge4|a|+24|b|R_*^2+4|c|t_0+60AR_*^4.
\]
Then
\begin{equation}\label{refm:eq:uniform-convexity}
\Psi_t''(r)\ge\frac{1}{2t},
\qquad 0<r<R,\quad 0<t<t_0,\quad R\le R_*.
\end{equation}
For the normalized restriction $\bar\nu_{t,\varepsilon}$
to $[\varepsilon,R]$, the Neumann diffusion
$\mathscr A_tf=f''-\Psi_t'f'$ satisfies
\[
\Gamma(f)=|f'|^2,
\qquad
\Gamma_2(f)=|f''|^2+\Psi_t''|f'|^2
\ge\frac{1}{2t}\Gamma(f).
\]
The Bakry--\'Emery criterion therefore gives
\begin{equation}\label{refm:eq:ref-LSI-epsilon}
\operatorname{Ent}_{\bar\nu_{t,\varepsilon}}(f^2)
\le4t\int_\varepsilon^R|f'|^2\,d\bar\nu_{t,\varepsilon};
\end{equation}
see \cite{BE} and
\cite[Appendix~A, Theorem~A1]{Schlichting}.
 Since
$e^{-\Psi_t(r)}\sim r^{n-1}$ as $r\downarrow0$ for fixed $t$,
we may let $\varepsilon\downarrow0$, first for smooth radial
functions and then by density, to obtain
\begin{equation}\label{refm:eq:ref-LSI}
\operatorname{Ent}_{\nu_t}(f^2)
\le4t\int_B|\nabla_{g_K}f|_{g_K}^2\,d\nu_t
\end{equation}
for every radial Sobolev function $f$.

Finally, by standard Schwarz rearrangement, the minimizer $u^*$ is radial. 
Since $f=\frac{u^*}{\phi_t}$ is radial and
$\int_Bf^2\,d\nu_t=1$, the reference-measure identity and
\eqref{refm:eq:ref-LSI} imply
\[
\begin{aligned}
\mathcal L(B,g_K,u^*,t)
&=4t\int_B|\nabla_{g_K}f|_{g_K}^2\,d\nu_t
-\operatorname{Ent}_{\nu_t}(f^2)
+\int_B\mathcal R_t(u^*)^2\,d\mu_K\\
&\ge\int_B\mathcal R_t(u^*)^2\,d\mu_K\\
&\ge\inf_{0\le r\le R}\mathcal R_t(r).
\end{aligned}
\]
Taking the infimum over $u$ and using
\eqref{refm:eq:R-lower} gives
\[
L(B,g_K,t)
\ge-n(n-1)Kt-\frac{n(n-1)}{3}K^2t^2-Ct^3.
\]
Together with the upper bound, this proves
\eqref{eq:fixed-ball-result}.
\end{proof}

\subsection{The complete space forms}\label{sec:complete-models}
\label{sec:hyperbolic}

\begin{proof}[Proof of Theorem \ref{prop:model-global}]
For $K=0$, the sharp Euclidean logarithmic Sobolev inequality
implies (cf.  \cite{Gross}, \cite{Weissler1978} and Lemma 8.1.7 in \cite{Topping})
\[
L(\mathbb R^n,g_E,t)=0.
\]
Suppose henceforth that $K\ne0$, and fix $R_0>0$ sufficiently
small for Theorem~\ref{thm:fixed-ball} to apply. We denote that
\[
\Omega=
\begin{cases}
M_K^n,&K>0,\\
B_K(o,R),\quad R\ge R_0,&K<0,
\end{cases}
\qquad
m(t)=L(\Omega,g_K,t).
\]
By attainment on compact model domains and Schwarz
rearrangement, we can choose a nonnegative minimizer
$u_t$ on $\Omega$ that is radial and nonincreasing about $o$. Moreover, Euler--Lagrange equation is
\begin{equation}\label{eq:model-unscaled-Euler}
-4t\Delta_{g_K}u_t-u_t\log u_t^2
=
\left(m(t)+n+\frac{n}{2}\log(4\pi t)\right)u_t.
\end{equation}
Introduce
\[
q=\frac{r}{\sqrt t},
\qquad
w_t(q)=t^{\frac{n}{4}}u_t(\sqrt t\,q),
\qquad
\sigma_{K,t}(q)=\frac{s_K(\sqrt t\,q)}{\sqrt t}.
\]
Write
\[
L_{K,t}=
\begin{cases}
\frac{\pi}{\sqrt{Kt}},&K>0,\\[1mm]
\frac{R}{\sqrt t},&K<0,
\end{cases}
\qquad
c_{K,t}(q)=\frac{\sigma_{K,t}'(q)}{\sigma_{K,t}(q)}.
\]
Then
\begin{equation}\label{eq:unified-ode}
-4\left(w_t''+(n-1)c_{K,t}w_t'\right)
-w_t\log w_t^2=\Lambda_tw_t,
\end{equation}
and
\begin{equation}\label{eq:unified-normalization}
|\mathbb S^{n-1}|
\int_0^{L_{K,t}}w_t^2\sigma_{K,t}^{n-1}\,dq=1,
\qquad
w_t'\le0.
\end{equation}
Equivalently, in the weak sense,
\begin{equation}\label{eq:unified-flux}
-4(\sigma_{K,t}^{n-1}w_t')'
=
\sigma_{K,t}^{n-1}w_t(\Lambda_t+\log w_t^2).
\end{equation}

\textbf{We first estimate the bound of $m(t)$.} On the sphere, the compact
Gagliardo--Nirenberg inequality gives, for $\int u^2\,d\mu_K=1$,
\[
\int |u|^{2+\frac{4}{n}}\,d\mu_K
\le C(1+E),
\qquad
E=\int|\nabla_{g_K}u|_{g_K}^2\,d\mu_K.
\]
Jensen's inequality therefore yields, in every dimension $n\ge2$,
\[
\int u^2\log u^2\,d\mu_K
\le\frac{n}{2}\log\int |u|^{2+\frac{4}{n}}\,d\mu_K
\le C+\frac{n}{2}\log(1+E).
\]
Consequently,
\[
\mathcal L(M_K^n,g_K,u,t)
\ge
4tE-\frac{n}{2}\log(1+E)
-n-\frac{n}{2}\log(4\pi t)-C.
\]
Since $4y-\frac n2\log y$ is bounded below on
$(0,\infty)$, the desired uniform lower bound follows
for $0<t\le t_0$.
For $K<0$, put $\kappa=\sqrt{-K}$. In geodesic polar coordinates centered at $o$, we have
\[
g_K=dr^2+
\left(\frac{\sinh(\kappa r)}{\kappa}\right)^2
g_{\mathbb S^{n-1}},
\qquad
d\mu_K=e^{h(r)}\,d\mu_E,
\]
where
\[
h(r)=(n-1)\log\left(
\frac{\sinh(\kappa r)}{\kappa r}
\right).
\]
Define $v=e^{\frac{h}{2}}u$, identifying the radial
coordinates on the hyperbolic and Euclidean balls. Then
\[
\int_{\mathbb R^n}v^2\,d\mu_E
=\int_{B_K(o,R)}u^2\,d\mu_K=1,
\]
where $v$ is extended by zero outside the Euclidean ball.
Since $u=e^{-\frac{h}{2}}v$ and both functions are radial,
\[
\int_{B_K(o,R)}|\nabla_{g_K}u|_{g_K}^2\,d\mu_K
=
\int_{\mathbb R^n}
\left|\nabla_{g_E}v-\frac12v\nabla_{g_E}h\right|_{g_E}^2
\,d\mu_E.
\]
Integrating  and by parts gives
\[
\begin{aligned}
4t\int_{B_K(o,R)}
|\nabla_{g_K}u|_{g_K}^2\,d\mu_K
={}&4t\int_{\mathbb R^n}
|\nabla_{g_E}v|_{g_E}^2\,d\mu_E+2t\int_{\mathbb R^n}
(\Delta_{g_E}h)v^2\,d\mu_E+t\int_{\mathbb R^n}
|\nabla_{g_E}h|_{g_E}^2v^2\,d\mu_E.
\end{aligned}
\]
Consequently,
\begin{equation}\label{eq:similar}
\begin{aligned}
\mathcal L(B_K(o,R),g_K,u,t)
={}&D_{E,t}(v)+\int_{\mathbb R^n}
\left(h+2t\Delta_{g_E}h
+t|\nabla_{g_E}h|_{g_E}^2\right)v^2\,d\mu_E,
\end{aligned}  
\end{equation}
where $D_{E,t}(v)=\mathcal L(\mathbb R^n,g_E,v,t)\ge0$
by the sharp Euclidean logarithmic Sobolev inequality
\cite{Weissler1978}.
Moreover, 
$
h'(r)=(n-1)\left(
\kappa\coth(\kappa r)-\frac1r\right)>0,
h''(r)=(n-1)\left(
\frac1{r^2}-\frac{\kappa^2}{\sinh^2(\kappa r)}
\right)>0.
$
Since $h(0)=0$ and
$
\Delta_{g_E}h=h''+\frac{n-1}{r}h'>0.
$
 Therefore
\[
\mathcal L(B_K(o,R),g_K,u,t)\ge0.
\]
For both the cases $K>0$ and $K<0$, domain monotonicity and the fixed-ball
expansion give
\[
m(t)\le L(B_K(o,R_0),g_K,t)\le Ct.
\]
Thus
\[
-C\le m(t)\le Ct\quad(K>0),
\qquad
0\le m(t)\le Ct\quad(K<0),
\]
and hence
\begin{equation}\label{eq:unified-lambda}
\Lambda_t:=m(t)+n+\frac{n}{2}\log(4\pi),
\qquad
|\Lambda_t|\le C_\Lambda.
\end{equation}

\textbf{We next construct a supersolution for $w_t$ and use it to derive an
exponential decay estimate for $W_t$.}
Fix
\[
0<\alpha\le\frac{1}{144(n+1)}.
\]
For each fixed $Q_0>1$ and sufficiently small $t$, monotonicity
and \eqref{eq:unified-normalization} imply
\[
w_t(Q_0)^2|\mathbb S^{n-1}|
\int_0^{Q_0}\sigma_{K,t}^{n-1}\,dq\le1.
\]
On this interval, $\sigma_{K,t}(q)\ge q$ if $K<0$, whereas
$\sigma_{K,t}(q)\ge\frac{2}{\pi}q$ if $K>0$ and
$Q_0\le\frac{L_{K,t}}{2}$. Thus
\[
w_t(Q_0)\le B_0:=C_nQ_0^{-\frac{n}{2}},
\]
with $C_n$ independent of $Q_0$, $t$, and $R$.
Choose $Q_0$ sufficiently large that
\begin{equation}\label{eq:barrier-margin}
-2\log(2B_0)-C_\Lambda\ge3,
\end{equation}
and then choose $t_0\le1$ sufficiently small that
$L_{K,t}\ge6Q_0$ for $0<t<t_0$.
For
\[
G_t(s)=-s\log s^2-\Lambda_ts,
\qquad G_t(0)=0,
\]
we then have
\[
G_t'(s)\ge1,
\qquad 0<s\le2B_0.
\]
Define
\[
b_0(q)=B_0e^{-\alpha(q-Q_0)^2},
\qquad
\mathscr P_{K,t}(z)
=-4\left(z''+(n-1)c_{K,t}z'\right)+G_t(z).
\]
A direct calculation gives
\begin{equation}\label{eq:unified-barrier}
\begin{aligned}
\frac{\mathscr P_{K,t}(b_0)}{b_0}
={}&(2\alpha-16\alpha^2)(q-Q_0)^2+8\alpha\\
&+8\alpha(n-1)c_{K,t}(q)(q-Q_0)
-2\log B_0-\Lambda_t.
\end{aligned}
\end{equation}
Since $\alpha<\frac{1}{8}$, it follows that
$\mathscr P_{K,t}(b_0)\ge2b_0$ wherever $q\ge Q_0$
and $c_{K,t}\ge0$. For $K<0$,
\[
c_{K,t}(q)=\sqrt{-Kt}\coth(\sqrt{-Kt}\,q)>0,
\]
so $b_0$ is a supersolution. In this case
we set $b_t=b_0$.

For $K>0$, write $L=L_{K,t}$. The preceding calculation
applies on the northern hemisphere, but the drift changes
sign on the southern hemisphere. We therefore define
\begin{equation}\label{eq:spherical-reflected-barrier}
b_t(q)=B_0\left(
e^{-\alpha(q-Q_0)^2}
+e^{-\alpha(2L-q-Q_0)^2}
\right).
\end{equation}
To verify its supersolution property on
$\frac{L}{2}\le q\le L$, put
\[
d=L-q,\qquad A=L-Q_0,\qquad y=\log b_t.
\]
Then
\[
b_t=2B_0e^{-\alpha(A^2+d^2)}\cosh(2\alpha Ad),
\]
so $b_t$ extends smoothly across the south pole, and
\[
\begin{aligned}
y'&=2\alpha\bigl(d-A\tanh(2\alpha Ad)\bigr),\\
y''&=-2\alpha+
4\alpha^2A^2\operatorname{sech}^2(2\alpha Ad).
\end{aligned}
\]
Using $0\le d\le\frac{L}{2}$, $A\le L$, and
$0\le\tanh s\le\min\{s,1\}$ for $s\ge0$, we obtain
\[
|y'|\le2\alpha\max\{d,A\}\le2\alpha L,
\qquad
|y''|\le2\alpha+4\alpha^2L^2,
\]
and
\[
\frac{|y'|}{d}\le2\alpha+4\alpha^2L^2,
\]
with the last bound understood by continuity at $d=0$.
Since
\[
c_{K,t}(q)=-\sqrt{Kt}\cot(\sqrt{Kt}\,d),
\qquad
|c_{K,t}(q)|\le\frac{1}{d},
\]
we have
\begin{equation}\label{eq:south-differential-bound}
\left|
\frac{b_t''+(n-1)c_{K,t}b_t'}{b_t}
\right|
\le2n\alpha+4(n+1)\alpha^2L^2.
\end{equation}
Moreover, $L\ge6Q_0$ implies $q-Q_0\ge\frac{L}{3}$ on
this hemisphere. Therefore
\begin{equation}\label{eq:south-log-bound}
b_t(q)\le2B_0e^{-\frac{\alpha L^2}{9}},
\qquad
-\log b_t(q)^2
\ge\frac{2\alpha}{9}L^2-2\log(2B_0).
\end{equation}
It follows from the choices of $\alpha$ and $B_0$ that
\[
\begin{aligned}
\frac{\mathscr P_{K,t}(b_t)}{b_t}
&\ge
\left(\frac{2\alpha}{9}-16(n+1)\alpha^2\right)L^2
-8n\alpha-2\log(2B_0)-\Lambda_t\\
&\ge\frac{\alpha}{9}L^2+3-8n\alpha>0.
\end{aligned}
\]
Thus $b_t$ is a supersolution on the southern hemisphere.
On the northern hemisphere, write $b_t=b_0+\delta_t$.
For $Q_0\le q\le\frac{L}{2}$,
\[
\frac{\delta_t}{b_0}
=e^{-4\alpha(L-Q_0)(L-q)}
\le e^{-\alpha L^2},
\qquad
0\le c_{K,t}(q)\le\frac{1}{Q_0}.
\]
Direct differentiation yields
\[
\left|\delta_t''+(n-1)c_{K,t}\delta_t'\right|
\le C(1+L^2)\delta_t.
\]
Since $b_0\le b_t\le2B_0$ and $G_t$ is increasing on this
range,
\[
\begin{aligned}
\mathscr P_{K,t}(b_t)
&\ge
\mathscr P_{K,t}(b_0)
-4\left|\delta_t''+(n-1)c_{K,t}\delta_t'\right|\\
&\ge
\left(2-C(1+L^2)e^{-\alpha L^2}\right)b_0
\ge b_0
\end{aligned}
\]
after decreasing $t_0$. Hence $b_t$ is a supersolution
on the entire sphere.

We now apply the weak comparison principle to show that
$w_t\le b_t$. Set $z=(w_t-b_t)_+$, so that
$z'=(w_t'-b_t')$ on $\{w_t>b_t\}$ and $z=0$
elsewhere. Since $w_t(Q_0)\le b_t(Q_0)$, we have
$z(Q_0)=0$. In the hyperbolic case, $z$ also vanishes
at the outer boundary by the Dirichlet condition. In the
spherical case, the south pole is an interior point of
the spherical cap $\{q>Q_0\}$ and produces no boundary term.
Since $\mathscr P_{K,t}(w_t)=0$ and
$\mathscr P_{K,t}(b_t)\ge0$, testing their difference
with $z$ and integrating by parts gives
\[
0\ge 4\int_{Q_0}^{L_{K,t}}|z'|^2\sigma_{K,t}^{n-1}\,dq
+\int_{Q_0}^{L_{K,t}}(G_t(w_t)-G_t(b_t))z\sigma_{K,t}^{n-1}\,dq.
\]
By radial monotonicity, $w_t\le B_0$ on the tail,
while $b_t\le2B_0$. Since $G_t'\ge1$ on $(0,2B_0]$,
we have
\[
(G_t(w_t)-G_t(b_t))z\ge z^2.
\]
Consequently,
\[
0\ge 4\int_{Q_0}^{L_{K,t}}|z'|^2\sigma_{K,t}^{n-1}\,dq
+\int_{Q_0}^{L_{K,t}}z^2\sigma_{K,t}^{n-1}\,dq\ge0.
\]
Both integrals must therefore vanish, so $z=0$ and
$w_t\le b_t$. Using the Gaussian bounds for $b_t$
in both geometries, we conclude that
\begin{equation}\label{eq:unified-gaussian-tail}
0\le w_t(q)\le2B_0e^{-\alpha(q-Q_0)^2},
\qquad Q_0\le q\le L_{K,t}.
\end{equation}
The constants are independent of sufficiently small $t$
and, when $K<0$, of the outer radius $R\ge R_0$.

We localize the argments to the fixed ball $B_K(o,R_0)$. The estimate \eqref{eq:unified-gaussian-tail}, together with
$\sigma_{K,t}(q)\le q e^{\sqrt{\max\{-K,0\}t}\,q}$
gives
\begin{equation}\label{eq:unified-mass-tail}
\varepsilon_t:=\int_{\Omega\setminus B_K(o,\frac{R_0}{2})}u_t^2\,d\mu_K\le Ce^{-c/t}.
\end{equation}
Choose a fixed smooth radial cutoff $0\le\chi\le1$, equal to one on $B_K(o,\frac{R_0}{2})$ and compactly supported in $B_K(o,R_0)$, and set
\[a_t=\int_\Omega\chi^2u_t^2\,d\mu_K,\qquad \widehat u_t=\frac{\chi u_t}{\sqrt{a_t}}.\]
Then $1-\varepsilon_t\le a_t\le1$, so $a_t\ge\frac{1}{2}$ for small $t$. Testing \eqref{eq:model-unscaled-Euler} with $\chi^2u_t$ yields that
\begin{equation}\label{eq:unified-cutoff-identity}
\mathcal L(B_K(o,R_0),g_K,\widehat u_t,t)=m(t)+\frac{1}{a_t}\int_\Omega u_t^2\left(4t|\nabla_{g_K}\chi|_{g_K}^2-\chi^2\log\chi^2\right)\,d\mu_K+\log a_t.
\end{equation}
Since $\log a_t\le0$, $|\nabla_{g_K}\chi|_{g_K}\le C$, and $0\le-\chi^2\log\chi^2\le\frac{1}{e}$, we obtain
\[L(B_K(o,R_0),g_K,t)\le\mathcal L(B_K(o,R_0),g_K,\widehat u_t,t)\le m(t)+Ce^{-c/t}.\]
Together with domain monotonicity, this proves
\begin{equation}\label{eq:unified-fixed-ball-comparison}
0\le L(B_K(o,R_0),g_K,t)-m(t)\le Ce^{-c/t}.
\end{equation}
For $K>0$, this already proves \eqref{eq:LK-model}.

For $K<0$, fix $0<t<t_0$, and set $m_R(t)=L(B_K(o,R),g_K,t)$. Domain monotonicity and $0\le m_R(t)\le Ct$ imply that
\[\ell(t):=\lim_{R\to\infty}m_R(t)\]
exists and is finite, with $L(M_k^n,g_K,t)\le\ell(t)$. For the reverse inequality, take any nonnegative admissible $f\in H^1(M_k^n)$ with $\int_{M_k^n}f^2\,d\mu_K=1$ and $f^2\log f^2\in L^1(M_k^n)$. Choose $\eta_j\in C_c^\infty(B_K(o,2j))$ such that $0\le\eta_j\le1$, $\eta_j=1$ on $B_K(o,j)$, and $|\nabla_{g_K}\eta_j|_{g_K}\le\frac{C}{j}$. Then $w_j=\eta_jf\to f$ in $H^1(M_k^n)$. Moreover,
$w_j^2\log w_j^2=\eta_j^2f^2\log f^2+\eta_j^2f^2\log\eta_j^2.$
We have
\[\int_{M_k^n}|\eta_j^2f^2\log\eta_j^2|\,d\mu_K\le\frac{1}{e}\int_{M_k^n\setminus B_K(o,j)}f^2\,d\mu_K\longrightarrow0.\]
Since $\|w_j\|_{L^2(M_k^n)}\to1$, the normalized functions $f_j=w_j/\|w_j\|_{L^2(M_k^n)}$ satisfy
\[\mathcal L(M,g_K,f_j,t)\longrightarrow\mathcal L(M,g_K,f,t).\]
Each $f_j$ is admissible on $B_K(o,2j)$ for sufficiently large $j$, so
\[\ell(t)\le m_{2j}(t)\le\mathcal L(M,g_K,f_j,t).\]
Letting $j\to\infty$ and taking the infimum over all admissible $f$ gives $\ell(t)\le L(M,g_K,t)$. Hence
\begin{equation}\label{eq:hyperbolic-exhaustion}
L(\mathbb H_K^n,g_K,t)=\lim_{R\to\infty}L(B_K(o,R),g_K,t).
\end{equation}
Since the constants in \eqref{eq:unified-fixed-ball-comparison} are independent of $R$, we may pass to the exhaustion limit for each fixed $t$ and obtain, in either geometry,
\[0\le L(B_K(o,R_0),g_K,t)-L(M_K^n,g_K,t)\le Ce^{-c/t}.\]
Finally, letting $t\downarrow0$, we conclude that
\[L(M_K^n,g_K,t)=-n(n-1)Kt-\frac{n(n-1)}{3}K^2t^2+o(t^2),\]
as claimed.
\end{proof}

\begin{rem}\label{Remark_subset}
For any nonempty open $U\subset M_K^n$, choose a fixed ball
$B_K(o,R_0)\Subset U$. Domain monotonicity and a local optimized Gaussian give
\[
 L(M_K^n,g_K,t)\le L(U,g_K,t)\le L(B_K(o,R_0),g_K,t)
       \le-n(n-1)Kt-\frac{n(n-1)}3K^2t^2+o(t^2).
\]
This shows that for any nonempty open $U\subset M_K^n$, we have
\begin{equation}
 L(U,g_K,t)
 =-n(n-1)Kt-\frac{n(n-1)}3K^2t^2+o(t^2)
 \qquad(t\downarrow0).
\end{equation}
\end{rem}

\section{Functional comparisons and the original rigidity applications}
\label{sec:functional-proofs}

We retain the following comparison of prescribed Gaussian families.

\begin{lem}\label{main_comparison}
	Let $(M^n,g)$ be an $n$-dimensional manifold and point $p\in  M^n$, and $M_K^n$  denote the \emph{complete simply connected} space
form of sectional curvature $K$.
		
	\noindent	(i)  If there exist a neighborhood $V_p$ of $p$ and  $u(x,t)\in \mathcal{B}_p(V)\cap \{u\mid \int_V u^2 d\mu =1\}$ satisfying 
	\begin{equation}\label{comparison_l}
		\mathcal{L}(V_p,g,u(x,t),t )\ge  L(M^n_K,g_K,t)-o(t)  \qquad(t\downarrow0),  
	\end{equation}
	then $$\operatorname{Sc}(p)\le n(n-1)K.$$	Moreover, if  there exist a neighborhood $V_p$ of $p$ and  $u(x,t)\in \mathcal{B}_p(V)\cap \{u\mid \int_V u^2 d\mu =1\}$ satisfying
	\begin{equation}\label{comparison_l_2d}
		\mathcal{L}(V_p,g,u(x,t),t )\ge  L(M^n_K,g_K,t)-o(t^2)  \qquad(t\downarrow0),
	\end{equation}
and if	
	we assume additionally that
	 $$\operatorname{Sc}(p)\ge n(n-1)K
\text{ and } \Delta \operatorname{Sc}(p)\ge 0,$$ then $$\operatorname{Sec}(p)=K.$$
	
	\noindent	(ii)  	If there exist a neighborhood $V_p$ of $p$ and $u(x,t)\in \mathcal{B}_p(V)\cap \{u\mid \int_V u^2 d\mu =1\}$ satisfying 
	\begin{equation}\label{comparison_v}
		\mathcal{W}(V_p,g, u(x,t),t )\ge  \mu (M^n_K,g_K,t)-o(t^2)   \qquad(t\downarrow0),
	\end{equation}
	then $|\operatorname{Rm}|^2(p)\le 2n(n-1)K^2$.	Moreover, if we assume additionally that  If $\operatorname{Sc}\ge n(n-1)K$ when $K>0$ or
$\operatorname{Sc}\le n(n-1)K$ when $K<0$ on $V$, then $\operatorname{Sec}_g\equiv K$ on $V$, then $$\operatorname{Sec}(p)=K.$$
\end{lem}
\begin{proof} (i)
	 It follows from  (\ref{comparison_l}) and (\ref{expansion_l_opt}) in Theorem \ref{expansion_opt} that $\operatorname{Sc}(p)\le n(n-1)K$. 
	 	
	By (\ref{comparison_l_2d}) and (\ref{expansion_l_opt}) in Theorem \ref{expansion_opt}, we conclude that
	\begin{equation}\label{q_1}
	\begin{aligned}
		-\operatorname{Sc}(p)t-\left(\Delta \operatorname{Sc}(p)+\frac{1}{6}|\operatorname{Rm}|^2(p)\right)t^2 
		\ge-n(n-1)Kt-\frac{n(n-1)}{3}K^2t^2-o(t^2) \qquad(t\downarrow0).
	\end{aligned}
	\end{equation}

 If we have $\operatorname{Sc}(p)\ge n(n-1)K$, then $\operatorname{Sc}(p)=\operatorname{Sc}_K(p_K)=n(n-1)K$. Since $\Delta \operatorname{Sc}(p)\ge 0$, 
we conclude from (\ref{q_1}) that 
\begin{equation}\label{q_2}
|\operatorname{Rm}|^2(p)\le |\operatorname{Rm}_K|^2(p_K)=2n(n-1)K^2.
\end{equation}
For $n=2$, the identity $\operatorname{Sc}=2\operatorname{Sec}$
already proves the assertion. For $n\ge3$, the curvature orthogonal decomposition
\begin{equation}\label{ortho_decomposition}
	|\mathrm{\operatorname{Rm}}|^2=\left|\frac{\operatorname{Sc}}{2 n(n-1)} g \odot g\right|^2+\left|\frac{1}{n-2} \stackrel{\circ}{\mathrm{\operatorname{Rc}} } \odot g\right|^2+\left|\operatorname { Weyl }\right|^2 
\end{equation}
and hence we have
\begin{equation}\label{q_3}
	|\mathrm{\operatorname{Rm}}|^2(p)\ge \left|\frac{\operatorname{Sc}}{2 n(n-1)} g \odot g\right|^2(p)=2n(n-1)K^2
\end{equation}
with the equality holds if and only if $g$ has the constant sectional curvature. Then $\operatorname{Sec}(p)=K$ follows by (\ref{q_2}) and (\ref{q_3}).

(ii)
	By (\ref{comparison_v}) and (\ref{expansion_nu_opt}), we conclude that
	$$
-\frac{1}{6}|\operatorname{Rm}|^2(p)t^2 \ge -\frac{n(n-1)}{3}K^2t^2-o(t^2)	 \qquad(t\downarrow0).
	$$
		 It follows that 
		\begin{equation}\label{q_4}
	|\operatorname{Rm}|^2(p)\le 2n(n-1)K^2  . 
		\end{equation}
		By the curvature orthogonal decomposition (\ref{ortho_decomposition}), we have
		\begin{equation}\label{q_5}
			|\mathrm{\operatorname{Rm}}|^2(p)\ge \left|\frac{\operatorname{Sc}}{2 n(n-1)} g \odot g\right|^2(p).
		\end{equation}
Then by (\ref{q_4}) and (\ref{q_5}) we conclude that $|\operatorname{Sc}(p)|\le n(n-1)|K|$.  If we assume additionally that $|\operatorname{Sc}(p)|\ge n(n-1)|K|$, then  $|\operatorname{Sc}(p)|=n(n-1)|K|$. Hence we have $|\mathrm{\operatorname{Rm}}|^2(p)=\left|\frac{\operatorname{Sc}}{2 n(n-1)} g \odot g\right|^2(p)$ and therefore $\operatorname{Sec}(p)=K$.
\end{proof}

\begin{proof}[Proof of Theorem~\ref{L_model_comparison}]
For any $p\in V$.
We take $u_{p,t}\in \mathcal{B}_p(V)\cap \{u\mid \int_V u^2 d\mu =1\}$. Thus we have
\begin{equation}\label{L_to_mu_comparison}
 \begin{split}
  \mathcal L (V,g,u_{p,t},t)&\ge L(V,g,t)\ge L(M_K^n,g_K,t)-o(t^2).
 \end{split}
\end{equation}
Then Lemma \ref{main_comparison} implies $\operatorname{Sc}(p)\le n(n-1)K.$ We have $\operatorname{Sc}(p)\le n(n-1)K$ by the assumption. Since $p$ is arbitrarily chosen, we get $\Delta R\equiv 0$ on $V$. Then Theorem~\ref{L_model_comparison} follows from  Lemma \ref{main_comparison}.
\end{proof}

\begin{proof}[Proof of Theorem~\ref{mu_model_comparison}]
For any $p\in V$.
We take $u_{p,t}\in \mathcal{B}_p(V)\cap \{u\mid \int_V u^2 d\mu =1\}$. Thus we have
\begin{equation}\label{L_to_mu_comparison}
 \begin{split}
  \mathcal W (V,g,u_{p,t},t)&\ge \mu(V,g,t)\ge \mu(M_K^n,g_K,t)-o(t^2).
 \end{split}
\end{equation}
 Then Theorem~\ref{mu_model_comparison} follows from  Lemma \ref{main_comparison}.
\end{proof}

Before presenting the proof of Theorem \ref{mu_rigidity_extension}, we need the following lemma.
 
 \begin{lem}\label{curvature_estimate}
 	Let $(M^n,g)$ be an $n$-dimensional manifold and $p\in \mathring{M^n}$.
 	If there exist neighborhood $V_p$ of $p$ and $u(x,t)\in \mathcal{B}_p(V_p)$   satisfying
 	\begin{equation}
 		\mathcal{W}(V_p,g, u,t)+	\gamma t^2\int_{V_p}|\operatorname{Rm}|^2u^2d\mu\ge -Qt^2 -o(t^2),
 	\end{equation}	
 	for $\gamma<\frac{1}{6}$,
 	then $|\operatorname{Rm}|^2(p)\le \frac{Q}{\frac{1}{6}-\gamma}$.
 \end{lem}

\begin{proof}
 Since $u\in\mathcal B_p(V_p)$ has fixed compact support and the Gaussian concentrates at $p$, its mass satisfies $\int_{V_p}u^2\,d\mu_g=1+O(t)$. Continuity of $|\operatorname{Rm}|^2$ and the same Gaussian localization give, with the integral finite on the fixed compact support,
$$
	\gamma t^2\int_{V_p}|\operatorname{Rm}|^2u^2d\mu=	\gamma |\operatorname{Rm}|^2(p)t^2+o(t^2).
$$
Combining this with \eqref{expansion_nu_opt}, we obtain
\begin{equation}\label{1.4_1}
	\mathcal{W}(V_p,g, u,t)+\gamma t^2\int_{V_p}|\operatorname{Rm}|^2u^2\,d\mu_g
=\left(-\frac{1}{6}+\gamma\right)|\operatorname{Rm}|^2(p)t^2+o(t^2).
\end{equation}
Then Lemma \ref{curvature_estimate} follows from (\ref{1.4_1}) directly.
\end{proof}

Now we give the proof of Theorem \ref{mu_rigidity_extension}.

\begin{proof}[Proof of Theorem \ref{mu_rigidity_extension}]
	Note that Theorem \ref{mu_rigidity_extension} (ii) just follows from Theorem \ref{mu_rigidity_extension} (i) by
	taking $\gamma=0$. So we only need to prove Theorem \ref{mu_rigidity_extension} (i).
For  any $p\in V$, and when $u(x,t)\in \mathcal{B}_p(V)$, by the assumption \eqref{mu_rigidity_extension_i} of Theorem \ref{mu_rigidity_extension}
we have  
$$
\frac{\gamma t^2\int_{V}|\operatorname{Rm}|^2u^2d\mu+\mathcal{W}(V,g, u,t)}{\int_V u(x,t)^2d\mu}\ge -Qt^2-o(t^2),
$$
and 
$$
\int_V u(x,t)^2d\mu= 1+O(t).
$$
Hence, we get
$$
\gamma t^2\int_{V}|\operatorname{Rm}|^2u^2d\mu+\mathcal{W}(V,g, u,t)\ge -Qt^2-o(t^2),
$$
Then Theorem \ref{mu_rigidity_extension} (i) follows from Theorem \ref{curvature_estimate}  directly.
\end{proof}

Before presenting the proof of Theorem \ref{mu_rigidity}, we need the following lemma.
	\begin{lem}\label{section}
	Let $(M^n,g)$ be an $n$-dimensional manifold and $p\in \mathring{M^n}$. 	\\
\noindent(i) If there exists a neighborhood $V_p$ of $p$ whose logarithmic Sobolev inequality only differs from  that of Euclidean case with $o(t)$, i.e.  there exist $\tau_0>0$ such that for all $f \in W_0^{1,2}(V_p)$, $\int_{V_p} f^2 d vol=1$, $p\in \operatorname{supp} \{f\}$ and $t\le\tau_0$,
\begin{equation}\label{comparison_iso_2_6_1st}
	\int_{V_p} f^2 \log f^2 d vol \le \int_{V_p}4t |\nabla f|^2 d vol-n-\frac{n}{2} \log (4 \pi t)+o(t),
\end{equation}
\textbf{then} the scalar curvature at $p$ satisfies
$$
\operatorname{Sc}(p)\le 0.
$$	

\noindent(ii) If there exists a neighborhood $V_p$ of $p$ whose logarithmic Sobolev inequality differs from the Euclidean case only by $o(t^2)$, i.e., there exists $\tau_0>0$ such that for all $f \in W_0^{1,2}(V_p)$ with $\int_{V_p} f^2 d vol=1$, $p\in \operatorname{supp} \{f\}$, and $t\le\tau_0$,
\begin{equation}\label{comparison_iso_2_6}
	\int_{V_p} f^2 \log f^2 d vol \le \int_{V_p}4t |\nabla f|^2 d vol-n-\frac{n}{2} \log (4 \pi t)+o(t^2),
\end{equation}
and if we additionally assume that
\begin{equation}\label{R_point_compare_6}
	\operatorname{Sc}(p) \ge 0, \quad \Delta \operatorname{Sc}(p) \ge 0, 
\end{equation}
\textbf{then} the sectional curvature at $p$ satisfies
$$\operatorname{Sec}(p)=0.$$
	\end{lem}
\begin{proof}
	\noindent(i) 
Just notice that for any $u(x,t)\in \mathcal{B}_p(V)$, we have for $f^2=\frac{u^2}{\int_V u(x,t)^2d\mu}$
$$
\frac{\mathcal{L}(V,g, u,t)}{\int_V u(x,t)^2d\mu}= \int_{V}4t |\nabla f|^2 d\mu-\int_{V} f^2 \log f^2 d\mu-n-\frac{n}{2} \log (4 \pi t),
$$
and
$$
\int_V u(x,t)^2d\mu= 1+O(t).
$$
It follows from \eqref{comparison_iso_2_6_1st} that
$$
\mathcal{L}(V,g, u,t)\ge -o(t).
$$
By  $O(t)$ term of  (\ref{expansion_l_opt}), we get $\operatorname{Sc}(p)\le 0$. 

	\noindent(ii)By \eqref{comparison_iso_2_6}, we have  $
	\mathcal{L}(V,g, u,t)\ge -o(t^2).
	$
	It follows from Lemma \ref{section} (i) and (\ref{R_point_compare_6}), we conclude that  $\operatorname{Sc}(p)= 0$. Then Lemma \ref{section}  follows from  (\ref{comparison_iso_2_6}) and  comparing the $O(t^2)$ term of (\ref{expansion_l_opt}). 

\end{proof}

Next we give the proof of Theorem  \ref{mu_rigidity}.
\begin{proof}[Proof of Theorem  \ref{mu_rigidity}]
From Lemma \ref{section}, we know that (\ref{log_sobolev}) implies $\operatorname{Sc}(x)\le 0$ for all $x\in V$. Combining this with (\ref{scalar_assump}), we get $\operatorname{Sc}\equiv 0$ on $V$. Then Theorem  \ref{mu_rigidity}  follows from  Lemma \ref{section} directly.
\end{proof}

\section{Schwarz rearrangement and the isoperimetric corollary}\label{sec:original-schwarz}

\begin{lem}\label{quantitative_Schwarz}

Let $(M,g)$ be an $n$-dimensional manifold and $V$ be an open nonempty subset of $M$.
Suppose that, for some $\beta_0>0$ and
$\nu>\frac{n-1}{n}$,
\begin{equation}\label{eq:eps-p}
\operatorname{I}(V,\beta)
\ge \operatorname{I}(M_K^n,\beta)-\varepsilon(\beta),
\qquad
\varepsilon(\beta)=o(\beta^\nu)
\end{equation}
for $0<\beta<\beta_0$.
Fix $p\in V$ and $q\in M_K^n$. Choose $\rho>0$ sufficiently
small that $B_g(p,3\rho)\Subset V$ and $B_K(q,3\rho)$ are
normal balls.
Choose an orthonormal basis of $T_pM$ and let
$x=(x^1,\ldots,x^n)$ be the corresponding geodesic normal
coordinates centered at $p$. For sufficiently small $t>0$,
define
\[
(P_t)_{ij}=\delta_{ij}-\frac{4t}{3}\operatorname{Rc}_{ij}(p),
\]
and
\[
\ell_t(x)=\left(\sum_{i,j=1}^n(P_t)_{ij}x^ix^j\right)^{\frac{1}{2}}
=\left(|x|^2-\frac{4t}{3}\operatorname{Rc}_{ij}(p)x^ix^j\right)^{\frac{1}{2}}.
\]
Choose a fixed nonincreasing smooth cutoff
$0\le\chi\le1$ on $[0,\infty)$, with $\chi=1$ on
$[0,\frac{\rho}{2}]$ and $\chi=0$ on $[\rho,\infty)$.
Let $u_t$ be the normalized test function
\begin{equation}\label{eq:ell-test}
u_t(x)=c_t(4\pi t)^{-\frac{n}{4}}
\exp\left(-\frac{\ell_t(x)^2}{8t}\right)\chi(\ell_t(x)),
\qquad
\int_Vu_t^2\,d\mu_g=1,
\end{equation}
extended by zero outside.
Let $u_t^*$ be its nonnegative Schwarz rearrangement of $u_t$ into
$B_K(q,2\rho)$, extended by zero to $M_K^n$. Then
\begin{equation}\label{eq:quant-PS-result}
\cL(V,g,u_t,t)
\ge
\cL(M_K^n,g_K,u_t^*,t)
-o\left(t^{\frac{n\nu-n+1}{2}}\right).
\end{equation}
\end{lem}

\begin{proof}
Throughout the proof, $c,C>0$ denote constants independent
of sufficiently small $t$ and of the level parameter.

\smallskip
\noindent\textit{The ellipsoidal levels and their volumes.}
Choose the cutoff $\chi$ in the statement so that, in addition,
$\chi(r)>0$ for $0\le r<\rho$. For sufficiently small $t$,
the matrix $P_t$ is positive definite and satisfies
$\frac{1}{2}I\le P_t\le2I$.
Set $Q_t=P_t^{-\frac{1}{2}}$, and write
$d\mu_g=J_g(x)\,dx$ in the chosen normal coordinates.
Then $Q_t\to I$ and $J_g(0)=1$.

For $0<a<\frac{\rho}{\sqrt t}$, define $s_t(a)$
to be the value of $u_t$ on the ellipsoid
$\ell_t(x)=\sqrt t\,a$, namely,
\begin{equation}\label{eq:level-s}
s_t(a):=c_t(4\pi t)^{-\frac{n}{4}}
e^{-\frac{a^2}{8}}\chi(\sqrt t\,a).
\end{equation}
The change of variables $x=\sqrt t\,Q_t y$ gives
\[
c_t^{-2}=(4\pi)^{-\frac{n}{2}}\det Q_t
\int_{|y|<\frac{\rho}{\sqrt t}}
e^{-\frac{|y|^2}{4}}\chi(\sqrt t\,|y|)^2
J_g(\sqrt t\,Q_t y)\,dy.
\]
Since $Q_t\to I$, $\chi(\sqrt t\,|y|)\to1$ for every
fixed $y$, and $J_g(\sqrt t\,Q_t y)\to1$, dominated
convergence with a Gaussian majorant shows that
$c_t\to1$. In particular, $c_t$ is uniformly bounded
above and below.
Differentiating \eqref{eq:level-s}, noting that $\chi'\le0$ by the assumption,  we obtain
\[
s_t'(a)=c_t(4\pi t)^{-\frac{n}{4}}e^{-\frac{a^2}{8}}
\left(-\frac{a}{4}\chi(\sqrt t\,a)
+\sqrt t\,\chi'(\sqrt t\,a)\right)<0
\]
for $0<a<\frac{\rho}{\sqrt t}$.  Thus $s_t$ parametrizes
all positive levels, decreasing from $\max u_t$ to zero,
and
\[
\{u_t>s_t(a)\}=\{\ell_t<\sqrt t\,a\}.
\]
Since $\chi$ and $\chi'$ are bounded, we also have
\begin{equation}\label{eq:level-deriv-bound}
|s_t'(a)|\le Ct^{-\frac{n}{4}}e^{-\frac{a^2}{8}}
(a+\sqrt t).
\end{equation}

Let
\[
\beta_t(s):=\operatorname{Vol}_g(\{u_t>s\}),
\qquad
b_t(a):=\beta_t(s_t(a)).
\]
 By change of variables $x=\sqrt t\,Q_t y$,
\[
b_t(a)=t^{\frac{n}{2}}\det Q_t
\int_{|y|<a}J_g(\sqrt t\,Q_t y)\,dy,
\]
and 
\[
b_t'(a)=t^{\frac{n}{2}}\det Q_t\,a^{n-1}
\int_{\mathbb S^{n-1}}J_g(\sqrt t\,aQ_t\theta)\,d\theta.
\]
Since $Q_t$, $\det Q_t$, and the volume density $J_g$
are uniformly controlled on the fixed normal neighborhood,
we obtain, uniformly for $0<a<\frac{\rho}{\sqrt t}$,
\begin{equation}\label{eq:dist-bounds}
c\,t^{\frac{n}{2}}a^n\le b_t(a)\le C\,t^{\frac{n}{2}}a^n,
\qquad
c\,t^{\frac{n}{2}}a^{n-1}\le b_t'(a)
\le C\,t^{\frac{n}{2}}a^{n-1}.
\end{equation}

By the normal-coordinate volume expansion and $Q_t\to I$,
we have
\[
\operatorname{Vol}_g(\{\ell_t<\rho\})
=\omega_n\rho^n(1+O(\rho^2)+O(t)),
\]
whereas
\[
\operatorname{Vol}_{g_K}(B_K(q,2\rho))
=\omega_n(2\rho)^n(1+O(\rho^2)).
\]
Thus, choosing $\rho$ sufficiently small at the outset
and then taking $t$ sufficiently small, we ensure that
\[
\operatorname{Vol}_g(\{\ell_t<\rho\})
<\min\left\{\beta_0,
\operatorname{Vol}_{g_K}(B_K(q,2\rho))\right\}.
\]
Consequently, the isoperimetric comparison
\eqref{eq:eps-p} applies to every positive superlevel
set of $u_t$, and the support of $u_t^*$ is compactly
contained in $B_K(q,2\rho)$.

For almost every positive regular value $s$, the coarea
formula gives
\[
-\beta_t'(s)=\int_{\{u_t=s\}}
\frac{1}{|\nabla_g u_t|_g}\,d\sigma_g.
\]
By Cauchy--Schwarz inequality,
\[
\operatorname{Area}_g(\{u_t=s\})^2
\le\left(\int_{\{u_t=s\}}|\nabla_g u_t|_g\,d\sigma_g\right)
(-\beta_t'(s)).
\]
Since $\{u_t>s\}\Subset V$ and
$\operatorname{Vol}_g(\{u_t>s\})=\beta_t(s)$, we have
\[
\operatorname{Area}_g(\{u_t=s\})
=\operatorname{Area}_g(\partial\{u_t>s\})
\ge\operatorname{I}(V,\beta_t(s)).
\]
 Applying the coarea
formula once more, we obtain
\begin{equation}\label{eq:coarea-cs}
\int_V|\nabla_g u_t|_g^2\,d\mu_g
\ge\int_0^{\max u_t}
\frac{\operatorname{I}(V,\beta_t(s))^2}
{-\beta_t'(s)}\,ds.
\end{equation}

We write
\[
\mathcal V_K(r):=\operatorname{Vol}_{g_K}(B_K(q,r)),
\]
and define $r_t(a)$ by $\mathcal V_K(r_t(a))=b_t(a)$.
Then we have
$u_t^*(r_t(a))=s_t(a)$. Since geodesic balls are
isoperimetric  in $M_K^n$,
we have
\[
\mathcal V_K'(r)=\operatorname{I}(M_K^n,\mathcal V_K(r)),
\qquad
r_t'(a)=\frac{b_t'(a)}{\operatorname{I}(M_K^n,b_t(a))}.
\]
Consequently, for any $0<a_1<a_2<\frac{\rho}{\sqrt t}$,
we have
\[
\int_{B_K(q,r_t(a_2))\setminus B_K(q,r_t(a_1))}
|\nabla_{g_K}u_t^*|_{g_K}^2\,d\mu_K
=\int_{a_1}^{a_2}
\frac{\operatorname{I}(M_K^n,b_t(a))^2|s_t'(a)|^2}{b_t'(a)}\,da.
\]
Using 
$\operatorname{I}(M_K^n,\beta)\le C\beta^{\frac{n-1}{n}}$
together with \eqref{eq:dist-bounds} and
\eqref{eq:level-deriv-bound}, we obtain
\[
\int_0^{\frac{\rho}{\sqrt t}}
\frac{\operatorname{I}(M_K^n,b_t(a))^2|s_t'(a)|^2}{b_t'(a)}\,da
\le\frac{C}{t}\int_0^\infty a^{n-1}(1+a)^2e^{-\frac{a^2}{4}}\,da<\infty.
\]
Moreover, differentiating $b_t(a)=\beta_t(s_t(a))$ gives
\[
-\beta_t'(s_t(a))=\frac{b_t'(a)}{|s_t'(a)|}.
\]
Thus, 
\begin{equation}\label{eq:rad-coarea}
\int_{M_K^n}|\nabla_{g_K}u_t^*|_{g_K}^2\,d\mu_K
=
\int_0^{\max u_t}
\frac{\operatorname{I}(M_K^n,\beta_t(s))^2}{-\beta_t'(s)}\,ds.
\end{equation}
By the standard Schwarz rearrangement, we have
\begin{align}
\int_Vu_t^2\,d\mu_g
&=\int_{M_K^n}(u_t^*)^2\,d\mu_K=1,
\label{1.1_1}\\
\int_Vu_t^2\log u_t^2\,d\mu_g
&=\int_{M_K^n}(u_t^*)^2\log (u_t^*)^2\,d\mu_K.
\label{1.1_2}
\end{align}

We now establish the quantitative energy comparison.
Since $\frac{\operatorname{I}(M_K^n,\beta)}{
\beta^{\frac{n-1}{n}}}\to 1$ as $\beta\downarrow0$ and
$\varepsilon(\beta)=o(\beta^\nu)$ with
$\nu>\frac{n-1}{n}$, we have
$
\frac{\varepsilon(\beta)}
{\operatorname{I}(M_K^n,\beta)}\longrightarrow0.
$
Thus, after decreasing the volume threshold,
$0\le\varepsilon(\beta)\le\operatorname{I}(M_K^n,\beta)$,
and \eqref{eq:eps-p} gives
\[
\operatorname{I}(V,\beta)^2
\ge\bigl(\operatorname{I}(M_K^n,\beta)-\varepsilon(\beta)\bigr)^2
\ge\operatorname{I}(M_K^n,\beta)^2
-2\operatorname{I}(M_K^n,\beta)\varepsilon(\beta).
\]
Combining this with \eqref{eq:coarea-cs} and
\eqref{eq:rad-coarea}, we obtain
\begin{equation}\label{eq:quant-PS}
4t\left(\int_{M_K^n}|\nabla_{g_K}u_t^*|_{g_K}^2\,d\mu_K-\int_V|\nabla_g u_t|_g^2\,d\mu_g\right)
\le8t\int_0^{\max u_t}\frac{\operatorname{I}(M_K^n,\beta_t(s))\varepsilon(\beta_t(s))}{-\beta_t'(s)}\,ds.
\end{equation}
Set
$\omega(\beta):=\frac{\varepsilon(\beta)}{\beta^\nu}$,
so that $\omega(\beta)\to0$ as $\beta\downarrow0$
and $\omega$ is bounded.
Using the change of variables $s=s_t(a)$, we obtain
\[
\begin{aligned}
\mathcal E_t
&:=8t\int_0^{\max u_t}
\frac{\operatorname{I}(M_K^n,\beta_t(s))
\varepsilon(\beta_t(s))}{-\beta_t'(s)}\,ds\\
&=8t\int_0^{\frac{\rho}{\sqrt t}}
\frac{\operatorname{I}(M_K^n,b_t(a))
\varepsilon(b_t(a))|s_t'(a)|^2}{b_t'(a)}\,da.
\end{aligned}
\]
By \eqref{eq:dist-bounds}, \eqref{eq:level-deriv-bound},
and $\operatorname{I}(M_K^n,\beta)\le
C\beta^{\frac{n-1}{n}}$, we have
\begin{equation}\label{eq:PS-scale}
0\le\mathcal E_t
\le Ct^{\frac{n\nu-n+1}{2}}
\int_0^{\frac{\rho}{\sqrt t}}
a^{n\nu}(1+a)^2e^{-\frac{a^2}{4}}
\omega(b_t(a))\,da.
\end{equation}
For every fixed $a>0$, \eqref{eq:dist-bounds} gives
\[
0\le b_t(a)\le Ct^{\frac n2}a^n\longrightarrow0,
\qquad
\omega(b_t(a))\longrightarrow0
\quad\text{as }t\downarrow0.
\]
Moreover, since $\omega$ is bounded and $n\nu>n-1$,
\[
0\le a^{n\nu}(1+a)^2e^{-\frac{a^2}{4}}
\omega(b_t(a))
\le Ca^{n\nu}(1+a)^2e^{-\frac{a^2}{4}},
\]
where the right-hand side belongs to $L^1(0,\infty)$.
Extending the integrand by zero outside
$(0,\frac{\rho}{\sqrt t})$, dominated convergence yields
\[
\lim_{t\downarrow0}
\int_0^{\frac{\rho}{\sqrt t}}
a^{n\nu}(1+a)^2e^{-\frac{a^2}{4}}
\omega(b_t(a))\,da=0.
\]
Consequently,
\[
\mathcal E_t=o\left(t^{\frac{n\nu-n+1}{2}}\right).
\]
Hence
\eqref{eq:quant-PS} implies
\[
\mathcal L(V,g,u_t,t)
\ge\mathcal L(M_K^n,g_K,u_t^*,t)
-o\left(t^{\frac{n\nu-n+1}{2}}\right),
\]
which proves \eqref{eq:quant-PS-result}.
\end{proof}

\begin{thm}\label{key_key}
Let $n\ge3$, $p\in\mathring M$, and suppose a neighborhood $V_p$ satisfies,
for small $\beta>0$,
\begin{equation}\label{comparison_iso_2}
 \inf_{\substack{\Omega\Subset V_p,\ p\in\Omega, \partial \Omega \text{ smooth}\\
                 \operatorname{Vol}_g(\Omega)=\beta}}
 \operatorname{Area}_g(\partial\Omega)
 \ge\operatorname{I}(M_K^n,\beta)-o(\beta^{\frac{n+3}{n}}).
\end{equation}
Then $\operatorname{Sc}(p)\le n(n-1)K$. If in addition
\begin{equation}\label{R_point_compare}
 \operatorname{Sc}(p)\ge n(n-1)K,\qquad\Delta\operatorname{Sc}(p)\ge0,
\end{equation}
then $\operatorname{Sec}(p)=K$. 
\end{thm}

\begin{proof}
Fix $p\in\mathring M$ and choose geodesic normal coordinates
$x=(x^1,\ldots,x^n)$ centered at $p$. For sufficiently small
$t>0$, set
\[
(P_t)_{ij}=\delta_{ij}-\frac{4t}{3}\operatorname{Rc}_{ij}(p),
\qquad
\ell_t(x)=\left((P_t)_{ij}x^ix^j\right)^{\frac{1}{2}},
\]
and consider the normalized ellipsoidal Gaussian
\[
u_t(x)=c_t(4\pi t)^{-\frac{n}{4}}
\exp\left(-\frac{\ell_t(x)^2}{8t}\right)\chi(\ell_t(x)),
\qquad
\int_Vu_t^2\,d\mu_g=1,
\]
where $\chi$ is the fixed cutoff in
Lemma~\ref{quantitative_Schwarz}. The choice of $P_t$
ensures that
\[
-\frac{\ell_t(x)^2}{8t}
=-\frac{|x|^2}{8t}
+\frac{1}{6}\operatorname{Rc}_{ij}(p)x^ix^j.
\]
Thus, on the region where $\chi=1$,
\begin{equation}\label{eq:ell-amplitude}
u_t^2=(4\pi t)^{-\frac{n}{2}}e^{-\frac{|x|^2}{4t}}\eta_t^2(x),
\qquad
\eta_t^2(x)=c_t^2
\exp\left(\frac{1}{3}\operatorname{Rc}_{ij}(p)x^ix^j\right).
\end{equation}
In particular,
\[
\eta_t^2(x)=c_t^2\left(
1+\frac{1}{3}\operatorname{Rc}_{ij}(p)x^ix^j
+O(|x|^4)\right).
\]
$\int_V u_t^2 d\mu =1$ gives
$c_t^2=1-\frac{1}{3}\operatorname{Sc}(p)t+O(t^2)$.
 Hence $u_t\in\mathcal B_p(V)$, and
Theorem~\ref{expansion_opt} yields
\begin{equation}\label{eq:ell-source-expansion}
\mathcal L(V,g,u_t,t)
=-\operatorname{Sc}(p)t
-\left(\Delta\operatorname{Sc}(p)
+\frac{1}{6}|\operatorname{Rm}|^2(p)\right)t^2
+o(t^2).
\end{equation}

We next compare $u_t$ with its Schwarz rearrangement $u_t^*$
on $M_K^n$. 
Since the proof of Lemma~\ref{quantitative_Schwarz} uses the isoperimetric comparison only for the positive superlevel sets of $u_t$, all of which contain $p$, assumption \eqref{comparison_iso_2} suffices. Applying the same argument with $\nu=\frac{n+3}{n}$, we obtain
\[
\mathcal L(V,g,u_t,t)
\ge\mathcal L(M_K^n,g_K,u_t^*,t)-o(t^2)\ge  L(M_K^n,g_K,t)-o(t^2).
\]
Then  it follows from Lemma \ref{main_comparison} that $\operatorname{sec}(p)=K$.
\end{proof}

\begin{proof}[Proof of Theorem~\ref{rigidity_iso_profile}]
For $n\ge3$, Lemma~\ref{key_key}, applied at each
$p\in V$, yields $Sc(p)\equiv n(n-1)K$. Then Theorem~\ref{rigidity_iso_profile} follows from Lemma~\ref{key_key}.

\end{proof}

\section{Appendix}
In the appendix, we prove that $L(\mathbb{H}_K^n, g_K, t)$ and 
$\mu(\mathbb{H}_K^n, g_K, t)$ are attainable in hyperbolic space. 
Although this result appears rather elementary, we were unable to find 
any reference in the literature that provides a proof. 
We therefore include a proof below.

\begin{thm}
\label{lem:hyperbolic-attainment}
For every $K=-\kappa^2<0$ and every fixed $t>0$, 
$L(\mathbb H_K^n,g_K,t)$ and
$\boldsymbol{\mu}(\mathbb H_K^n,g_K,t)$ are attained.
More precisely, there exists a nonnegative, radial,
nonincreasing function $u_t\in H^1(\mathbb H_K^n)$ satisfying
\[
\int_{\mathbb H_K^n}u_t^2\,d\mu_K=1,
\qquad
u_t^2\log u_t^2\in L^1(\mathbb H_K^n),
\]
such that
\[
\begin{aligned}
L(\mathbb H_K^n,g_K,t)
&=\mathcal L(\mathbb H_K^n,g_K,u_t,t),\\
\boldsymbol{\mu}(\mathbb H_K^n,g_K,t)
&=\mathcal W(\mathbb H_K^n,g_K,u_t,t).
\end{aligned}
\]
\end{thm}

\begin{proof}
Since $\operatorname{Sc}_{g_K}=n(n-1)K$ is constant on
$\mathbb H_K^n$ and the admissible functions are normalized,
we have
\[
\boldsymbol{\mu}(\mathbb H_K^n,g_K,t)
=
L(\mathbb H_K^n,g_K,t)+n(n-1)Kt.
\]
Thus it suffices to prove that
$L(\mathbb H_K^n,g_K,t)$ is attained.
By Schwarz rearrangement on $\mathbb H_K^n$, a minimizing
sequence may be chosen to consist of compactly supported,
nonnegative, radial, nonincreasing functions.

\smallskip
\noindent\textit{Step 1. Radial reduction and the density transformation.}

Write $K=-\kappa^2$ with $\kappa>0$, and let $(r,\theta)$ be
geodesic polar coordinates centered at $o$, so that
\[
g_K=dr^2+
\left(\frac{\sinh(\kappa r)}{\kappa}\right)^2
g_{\mathbb S^{n-1}},
\qquad
d\mu_K=J_K(r)\,d\mu_E,
\]
where $g_E$ denotes the Euclidean metric on $\mathbb R^n$ and
\[
J_K(r):=e^{h(r)},
\qquad
h(r):=(n-1)\log\left(
\frac{\sinh(\kappa r)}{\kappa r}
\right).
\]

We absorb the hyperbolic volume density into the function.
For a normalized radial function $u$, set
\[
v=J_K^{\frac12}u=e^{\frac h2}u,
\]
so that
\[
\int_{\mathbb R^n}v^2\,d\mu_E
=
\int_{\mathbb H_K^n}u^2\,d\mu_K
=1.
\]
Similar to \eqref{eq:similar}
\begin{equation}\label{eq:hyperbolic-attainment-identity}
\mathcal L(\mathbb H_K^n,g_K,u,t)
=
D_{E,t}(v)
+\int_{\mathbb R^n}V_{K,t}(r)v^2\,d\mu_E,
\end{equation}
where
\[
\begin{aligned}
D_{E,t}(v)
:={}&4t\int_{\mathbb R^n}
|\nabla_{g_E}v|_{g_E}^2\,d\mu_E-\int_{\mathbb R^n}v^2\log v^2\,d\mu_E
-n-\frac n2\log(4\pi t),
\end{aligned}
\]
and
\begin{equation}\label{eq:hyperbolic-effective-potential}
V_{K,t}(r)
:=
h(r)+2t\,\Delta_{g_E}h(r)
+t|\nabla_{g_E}h(r)|_{g_E}^2.
\end{equation}
Thus the attainment problem on $\mathbb H_K^n$ is reduced
to a Euclidean minimization problem with a confining potential.

\smallskip
\noindent\textit{Step 2. Coercivity and uniform estimates.}

A direct computation gives
\[
h'(r)=(n-1)\left(
\kappa\coth(\kappa r)-\frac1r
\right)>0,
\]
and
\[
h''(r)=(n-1)\left(
\frac1{r^2}-\frac{\kappa^2}{\sinh^2(\kappa r)}
\right)>0.
\]
Consequently,
\[
\Delta_{g_E}h
=h''+\frac{n-1}{r}h'>0,
\qquad
V_{K,t}(r)\ge h(r)\ge0.
\]
Moreover,
\[
\lim_{r\to\infty}\frac{h(r)}r=(n-1)\kappa>0.
\]
Hence there exist constants $c_0,C_0>0$, depending only
on $n$ and $K$, such that
\begin{equation}\label{eq:h-coercive}
h(r)\ge c_0r-C_0,\qquad r\ge0.
\end{equation}
We also have
\[
0<\coth x-\frac1x<1,\qquad x>0,
\]
and therefore
\begin{equation}\label{eq:h-gradient-bound}
|\nabla_{g_E}h|_{g_E}=h'(r)
\le(n-1)\kappa.
\end{equation}
In particular, $V_{K,t}(r)\to+\infty$ as $r\to\infty$.

Fix $t>0$, and let $(u_j)$ be a normalized radial minimizing
sequence for $L(\mathbb H_K^n,g_K,t)$. Set
\[
v_j=J_K^{\frac12}u_j,
\qquad
\int_{\mathbb R^n}v_j^2\,d\mu_E=1.
\]
By the Euclidean sharp logarithmic Sobolev inequality,
$D_{E,t}(v_j)\ge0$. Hence
\eqref{eq:hyperbolic-attainment-identity}, together with
$V_{K,t}\ge h\ge0$, yields
\begin{equation}\label{eq:uniform-deficit-potential}
\sup_jD_{E,t}(v_j)<\infty,
\qquad
\sup_j\int_{\mathbb R^n}h(r)v_j^2\,d\mu_E<\infty.
\end{equation}
Combining the second estimate with \eqref{eq:h-coercive},
we obtain
\[
\int_{|x|\ge R}v_j^2\,d\mu_E
\le\frac{C}{c_0R-C_0}
\]
for sufficiently large $R$, uniformly in $j$. Thus
\begin{equation}\label{eq:tightness-v}
\lim_{R\to\infty}\sup_j
\int_{\mathbb R^n\setminus B_R}v_j^2\,d\mu_E=0,
\end{equation}
and, in addition,
\begin{equation}\label{eq:first-moment-v}
\sup_j\int_{\mathbb R^n}|x|v_j^2\,d\mu_E<\infty.
\end{equation}
Finally, the scale-optimized Euclidean logarithmic Sobolev
inequality gives
\[
\int_{\mathbb R^n}v^2\log v^2\,d\mu_E
\le C_n+\frac n2\log\left(
1+\int_{\mathbb R^n}
|\nabla_{g_E}v|_{g_E}^2\,d\mu_E
\right)
\]
for normalized $v\in H^1(\mathbb R^n)$. Writing
\[
E_j=\int_{\mathbb R^n}
|\nabla_{g_E}v_j|_{g_E}^2\,d\mu_E,
\]
we obtain
\[
D_{E,t}(v_j)
\ge4tE_j-\frac n2\log(1+E_j)-C_{n,t}.
\]
Since $t>0$ is fixed, the right-hand side tends to
$+\infty$ as $E_j\to\infty$. Therefore
\begin{equation}\label{eq:uniform-H1}
\sup_j\int_{\mathbb R^n}
|\nabla_{g_E}v_j|_{g_E}^2\,d\mu_E<\infty.
\end{equation}
In particular, $(v_j)$ is bounded in $H^1(\mathbb R^n)$.

\smallskip
\noindent\textit{Step 3. Compactness and attainment.}

By weak compactness in $H^1(\mathbb R^n)$ and the Rellich
theorem, after passing to a subsequence, there exists a
nonnegative radial function $v\in H^1(\mathbb R^n)$ such that
\[
v_j\rightharpoonup v\quad\text{in }H^1(\mathbb R^n),
\qquad
v_j\to v\quad\text{in }L^2_{\mathrm{loc}}(\mathbb R^n).
\]
Combining the local convergence with \eqref{eq:tightness-v}
gives
\[
v_j\to v\quad\text{strongly in }L^2(\mathbb R^n),
\qquad
\int_{\mathbb R^n}v^2\,d\mu_E=1.
\]
We next verify convergence of the entropy integrals. Write
$\rho_j=v_j^2$ and $\rho=v^2$. The uniform Sobolev bound
controls the positive parts of $\rho_j\log\rho_j$, while
\eqref{eq:first-moment-v} and
\[
(\rho\log\rho)_-
\le a|x|\rho+e^{-a|x|-1},
\qquad a>0,
\]
control the negative parts uniformly at infinity. More
precisely, one first lets $R\to\infty$ with $a$ fixed and
then lets $a\downarrow0$. Thus the entropy densities are
uniformly integrable. By local convergence in measure and
Vitali's theorem,
\begin{equation}\label{eq:entropy-convergence}
\int_{\mathbb R^n}v_j^2\log v_j^2\,d\mu_E
\longrightarrow
\int_{\mathbb R^n}v^2\log v^2\,d\mu_E.
\end{equation}
By weak lower semicontinuity of the Dirichlet integral and
Fatou's lemma for the nonnegative potential $V_{K,t}$,
\[
\begin{aligned}
&D_{E,t}(v)+\int_{\mathbb R^n}V_{K,t}v^2\,d\mu_E\le
\liminf_{j\to\infty}
\left(
D_{E,t}(v_j)+
\int_{\mathbb R^n}V_{K,t}v_j^2\,d\mu_E
\right).
\end{aligned}
\]
Hence $v$ is a normalized minimizer of the transformed
functional.

Returning to the hyperbolic space, define
\[
u_t=J_K^{-\frac12}v=e^{-\frac h2}v.
\]
Since $u_j=J_K^{-\frac12}v_j$ and the density transformation
is an isometry in $L^2$, we have
\[
u_j\to u_t
\quad\text{strongly in }L^2(\mathbb H_K^n,d\mu_K).
\]
The class of radial nonincreasing functions is closed under
this convergence, so $u_t$ is also radial and nonincreasing.
Moreover,
$
\int_{\mathbb H_K^n}u_t^2\,d\mu_K=1.
$
By \eqref{eq:h-gradient-bound},
\[
\begin{aligned}
\int_{\mathbb H_K^n}|\nabla_{g_K}u_t|_{g_K}^2\,d\mu_K
&=\int_{\mathbb R^n}
\left|\nabla_{g_E}v-\frac12v\nabla_{g_E}h\right|_{g_E}^2
\,d\mu_E\\
&\le
2\int_{\mathbb R^n}|\nabla_{g_E}v|_{g_E}^2\,d\mu_E
+\frac12(n-1)^2\kappa^2
\int_{\mathbb R^n}v^2\,d\mu_E<\infty.
\end{aligned}
\]
Furthermore, Fatou's lemma and
\eqref{eq:uniform-deficit-potential} give
$
\int_{\mathbb R^n}h(r)v^2\,d\mu_E<\infty.
$
Together with \eqref{eq:entropy-convergence}, this implies
$
u_t^2\log u_t^2\in L^1(\mathbb H_K^n).
$
Thus $u_t$ belongs to the admissible class and attains
$L(\mathbb H_K^n,g_K,t)$.
 This completes the proof.
\end{proof}

\end{document}